\documentclass[11pt,reqno,oneside]{amsart}

\usepackage{amsmath,amssymb}
\usepackage{bbm}
\usepackage{color}
\usepackage{tikz}
\usepackage{cite}
\usepackage[margin=1in]{geometry}
\usepackage{verbatim}

\usepackage[colorlinks, linkcolor={black}, citecolor={black}, urlcolor={black}]{hyperref}

\usepackage{bm}

\theoremstyle{plain}
\newtheorem{theorem}{Theorem}[subsection]
\newtheorem{prop}[theorem]{Proposition}
\newtheorem{lemma}[theorem]{Lemma}
\newtheorem{coro}[theorem]{Corollary}
\newtheorem{fact}[theorem]{Fact}

\theoremstyle{definition}
\newtheorem{definition}[theorem]{Definition}
\newtheorem{remark}[theorem]{Remark}

\newtheorem{example}[theorem]{Example}

\newcommand{\C}{{\mathbb C}}
\newcommand{\D}{{\mathbb D}}

\newcommand{\N}{{\mathbb N}}
\newcommand{\Q}{{\mathbb Q}}
\newcommand{\R}{{\mathbb R}}
\newcommand{\T}{{\mathbb T}}
\newcommand{\X}{{\mathbb X}}
\newcommand{\Y}{{\mathbb Y}}
\newcommand{\Z}{{\mathbb Z}}
\newcommand{\mc}{\mathcal}
\newcommand{\im}{{\mathrm{i}}}

\newcommand{\Refl}{\operatorname{R}}

\newcommand{\me}{\mathrm{e}}
\newcommand{\mi}{\mathrm{i}}

\newcommand{\tr}{\operatorname{Tr}}

\newcommand{\ind}{\operatorname{ind}}
\newcommand{\Leb}{\operatorname{Leb}}

\newcommand{\pf}{\operatorname{p}}
\newcommand{\per}{{\operatorname{per}}}
\newcommand{\halfZ}{{\tfrac{1}{2}\Z}}

\renewcommand{\Im}{{\mathrm{Im}}}

\newcommand{\trigpoly}{{\mathrm{TP}}}

\newcommand{\SL}{{\operatorname{SL}}}

\newcommand{\set}[1]{\left\{#1\right\}}

\usepackage{mathtools}

\DeclarePairedDelimiter\norm{\lVert}{\rVert}

\newcommand{\Sf}{\operatorname{s}}
\newcommand{\MSf}{\operatorname{s'}} 
\newcommand{\Star}{(\ast)}

\numberwithin{equation}{subsection}

\begin{document}

\title[Operators Generated by Product systems]{Spectral Characteristics of Schr\"odinger Operators Generated by Product Systems}

\author[D.\ Damanik]{David Damanik}
\address{Department of Mathematics, Rice University, Houston, TX~77005, USA}
\email{damanik@rice.edu}
\thanks{D.D.\ was supported in part by NSF grants DMS--1700131 and DMS--2054752, an Alexander von Humboldt Foundation research award, and Simons Fellowship $\# 669836$}

\author[J.\ Fillman]{Jake Fillman}
\address{Department of Mathematics, Texas State University, San Marcos, TX 78666, USA}
\email{fillman@txstate.edu}
\thanks{J.F.\ was supported in part by Simons Collaboration Grant \#711663.}

\author[P.\ Gohlke]{Philipp Gohlke}
\address{Fakult\"at f\"ur Mathematik, Universit\"at Bielefeld,  Postfach 100131, 33501 Bielefeld, Germany}
\email{pgohlke@math.uni-bielefeld.de}
\thanks{P.G. acknowledges support by the German Research Foundation (DFG) via
the Collaborative Research Centre (CRC 1283)}

\begin{abstract}
Motivated by the question of what spectral properties of dynamically defined Schr\"odinger operators may be preserved under periodic perturbations, we study ergodic Schr\"odinger operators defined over product dynamical systems in which one factor is periodic and the other factor is either a subshift over a finite alphabet or an irrational rotation of the circle. The scenario given by a periodic background potential corresponds to a separable structure in which the sampling function is the sum of two pieces, each of which depends only on a single factor of the product system. However, in each case that we study, our methods apply more generally to sampling functions that allow nontrivial dependencies between the product factors.

In the case in which one factor is a Boshernitzan subshift, we prove that either the resulting operators are  periodic or the resulting spectra must be Cantor sets. The main ingredient is a suitable stability result for Boshernitzan's criterion under taking products. We also discuss the stability of purely singular continuous spectrum, which, given the zero-measure spectrum result, amounts to stability results for eigenvalue exclusion. In particular, we examine situations in which the existing criteria for the exclusion of eigenvalues are stable under periodic perturbations. As a highlight of this, we show that any simple Toeplitz subshift over a binary alphabet exhibits uniform absence of eigenvalues on the hull for any periodic perturbation whose period is commensurate with the coding sequence. This is new, even in the case in which the periodic background vanishes entirely. In the case of a full shift, we give an effective criterion to compute exactly the spectrum of a random Anderson model perturbed by a potential of period two, and we further show that the naive generalization of this criterion does not hold for period three. Next, we consider quasi-periodic potentials with potentials generated by trigonometric polynomials with periodic background. We show that the quasiperiodic cocycle induced by passing to blocks of period length is subcritical when the coupling constant is small and supercritical when the coupling constant is large. Thus, the spectral type is absolutely continuous for small coupling and pure point (for a.e.\ frequency and phase) when the coupling is large.
\end{abstract}


\maketitle

\setcounter{tocdepth}{2}
\tableofcontents


\hypersetup{colorlinks, linkcolor={black!30!blue}, citecolor={black!30!blue}, urlcolor={black!30!blue}}

\section{Introduction}
\subsection{Setting and Motivation}
We study Schr\"odinger operators in $\ell^2(\Z)$, that is, operators of the form
\begin{equation} \label{eq:HVdef}
H_V = \Delta + V,
\end{equation}
where the potential $V:\Z \to \R$ is bounded. There has been extensive work done for such operators; the reader may use \cite{CFKS, Damanik2017ETDS, DF1, DF2, Jitomirskaya2007ESO, MarxJito2017ETDS, Teschl2000MSS} and references therein as guides to the literature. In this paper we are interested in questions that lead one to the consideration of products of dynamical systems.

Let us explain how these product systems arise naturally. In many applications of interest, the potential $V$ is given by the sum of two terms,
\begin{equation} \label{eq:HVW}
H = \Delta+V =\Delta+V_1+V_2.
\end{equation}
For instance, one may consider the situation in which $V_1$ is random and $V_2$ is periodic, which supplies a model of a crystal with random impurities; compare \cite{AK19, CH94, GK13, KSS98b, KSS98, K99, K99Erratum, ST20, V02} for a partial list of papers studying this model. Another class of examples in the closely related continuum setting is given by sums of two periodic potentials with incommensurate frequencies, which provide the simplest examples of quasi-periodic potentials; compare \cite{FedotovKlopp2005TAMS, FroSpeWit1990CMP, SoretsSpencer1991CMP} for an incomplete list. This is but a partial list of potential settings; other recent papers consider more general additive perturbations of random \cite{CaiDuaKle2021preprint, DGpreprint} and quasiperiodic \cite{WXYZZpreprint} potentials.

In both examples mentioned in the previous paragraph the two summands have additional structure --- they are dynamically defined, in the sense that they are obtained by sampling along the orbit of a dynamical system (discrete-time in the first example and continuous-time in the second example). 

Thus, we will be interested in the case where $V_1, V_2$ take the following form,
\begin{equation}
V_j(n) = V_{j,x^{(j)}}(n) = f_j(S_j^n x^{(j)}), \quad j=1,2, \ n \in \Z,
\end{equation}
where $x^{(j)} \in \X_j$, a compact metric space, $S_j : \X_j \to \X_j$ is a homeomorphism, and $f_j:\X_j \to \R$ is continuous. Indeed, in the random case ($j=1$), one may take $\X_1$ to be a suitable sequence space, $S_1$ the left shift thereupon, and $f_1$ evaluation at the origin, and in the periodic case ($j=2$), we may take $\X_2 = \Z_p$ where $p$ is the period of $V_2$, $S_2 x^{(2)}=x^{(2)} +1 \, \mathrm{mod} \, p$, and $f_2(x^{(2)}) = V_2(\tilde x^{(2)})$, where $\tilde{x}^{(2)}$ denotes any representative of the residue class $x^{(2)}$.    
 
Clearly then, $V = V_1+V_2$ admits a description in terms of the product system $\Omega = \X_1 \times \X_2$, $T = S_1 \times S_2$. In particular,
\begin{equation} \label{eq:intro:VomegaDef}
V(n) = V_\omega(n) = f(T^n\omega),
\end{equation}
where $\omega = (x^{(1)},x^{(2)})$ and 
\begin{equation} \label{eq:intro:sepFct}
f(x^{(1)},x^{(2)}) = f_1(x^{(1)})+f_2(x^{(2)}).
\end{equation}

On one hand, the choice of sampling function as in \eqref{eq:intro:sepFct} is completely natural given the motivating scenario (a potential given by a sum of a dynamically defined potential and periodic background potential). On the other hand, as soon as one contextualizes the problem with product systems, it becomes natural to consider more general functions $f \in C(\Omega)$ that allow for more significant interactions between the factors. For instance, one may also consider sampling functions of the form 
\begin{equation} \label{eq:intro:sepFctmult}
f(x^{(1)},x^{(2)}) = f_1(x^{(1)})\cdot f_2(x^{(2)}),
\end{equation}
which corresponds to a periodic multiplicative modulation of a given potential. In general, it is more difficult to study periodic multiplicative perturbations than periodic additive perturbations. However, this is a natural outcome of our framework. One instance of such periodic multiplicative modifications comes from the \emph{trimmed Anderson model}, which corresponds to choosing $f$ as in \eqref{eq:intro:sepFctmult} with $f_2(k) = \delta_{k \, \mathrm{mod} \, p, 0}$; see \cite{ElgartKlein2014JST, ElgartSodin2017JST, KirschKrishna2020JST, Rojas-Molina2014OAM} and references therein.

To keep the length and complexity of the introduction in check, we formulate results in the two specific settings mentioned above, but emphasize that each of these results will be deduced as a consequence of a more general statement that allows one to consider quite general functions on the product space(s) that do not need any separable structure.

\subsection{Main Results}
We begin the general study of product systems by looking in detail at three particular instances of the general problem. In each instance, one of the factors will be chosen to be a finite shift on a cyclic space. We then consider results when the other factor is a minimal aperiodic subshift satisfying the Boshernitzan criterion, a Bernoulli shift, or an irrational rotation of the circle. Let us mention that the three settings we consider allow us to study periodic decorations of the Fibonacci Hamiltonian, the Bernoulli Anderson model, and the almost-Mathieu operator, which are the three of the most heavily-studied familes of ergodic one-dimensional Schr\"odinger operators. As a byproduct of this approach, we will obtain information about potentials generated by adding a periodic background to a model that is understood.   As noted above, this scenario corresponds to sampling functions that are separable as in \eqref{eq:intro:sepFct}; however, in each of the three instances, we are able to prove results that cover a larger class of continuous functions on the product space.

Although we investigate product systems in which one factor is a shift on a cyclic group, we emphasize that there are other natural choices for factors generating product systems of interest in mathematical physics that we hope will be addressed in future work.

In the first case, we consider subshifts satisfying the Boshernitzan condition and full shifts over finite alphabets. A subshift over a finite alphabet $\mathcal{A}$ is a compact, shift-invariant subset $\X \subseteq \mathcal{A}^\Z$. Here, $\mathcal{A}$ is given the discrete topology, and the shift $S:\X \to \X$ is given by $[Sx]_n = x_{n+1}$.  If $\X$ is minimal, it is said to satisfy \emph{Boshernitzan's criterion} if there exists an $S$-invariant probability measure $\nu$ on $\X$ with the property that 
\begin{equation}
\limsup_{n\to\infty} n \cdot \min\{\nu\{x \in \X : x_0 \cdots x_{n-1} = u \} : u \in \mathcal{L}_n(\X) \} >0,
\end{equation}
where $\mathcal{L}_n(\X)$ denotes the set of all words of length $n$ that occur in sequences in $\X$.

Motivated by our discussion above, we prove that zero-measure Cantor spectrum is stable under periodic perturbations for potentials of Boshernitzan type. Given a subshift $(\X,S)$ and a function $f_1:\X \to \R$, the potential $V_x$ is given by
\begin{equation} \label{eq:intro:Vxdef}
V_x(n) = f_1(S^nx),
\end{equation}
and the associated Schr\"odinger operator is denoted by $H_x$. We say that $f_1$ is \emph{locally constant} if 
\begin{equation} \label{eq:intro:locConstfdef}
f_1(x) = g(x_nx_{n+1}\ldots x_{n+k-1})
\end{equation}
for some $n \in \Z$, $k \in \N$, and $g:\mc A^k \to \R$.
\begin{theorem} \label{t:bosh:main}
Suppose $(\X,S)$ is a minimal subshift satisfying the Boshernitzan condition, $f_1:\X \to \R$ is locally constant, and $V_\per$ is periodic. One has the following dichotomy: Either $V_x$ defined in \eqref{eq:intro:Vxdef} is periodic for all $x \in \X$ or, for every $x \in \X$, $\sigma(H_x+V_\per)$ is a Cantor set of zero Lebesgue measure.
\end{theorem}

As discussed before, we will see later on that one can deduce Cantor spectrum in a more general setting than the one proposed in Theorem~\ref{t:bosh:main}. We refer the reader to Section~\ref{sec:persubshift} for details, but let us highlight one other outcome of the approach.  

\begin{theorem} \label{t:bosh:permult}
Suppose $(\X,S)$ is a minimal subshift satisfying the Boshernitzan condition, $f_1:\X \to \R$ is locally constant, and $\lambda_\per:\Z \to \R$ is periodic. One has the following dichotomy: For each $x \in \X$, either $\lambda_\per V_x$ with $V_x$ defined in \eqref{eq:intro:Vxdef} is periodic  or $\sigma(\Delta + \lambda_\per V_x)$ is a Cantor set of zero Lebesgue measure.
\end{theorem}

Notice that there is a slight difference between Theorems~\ref{t:bosh:main} and \ref{t:bosh:permult} with regard to their dependence on $x \in \X$. Namely, in Theorem~\ref{t:bosh:main}, the dichotomy is chosen globally, while the dichotomy in Theorem~\ref{t:bosh:permult} holds for each individual $x \in \X$. Moreover, one cannot avoid this distinction; there exist aperiodic subshifts $\X$ and periodic sequences $\lambda_\per$ such that $\lambda_\per V_x$ is periodic for some but not all $x \in \X$ (see Remark~\ref{rem:boshper:lambdaperXperiodDoubling}).

As soon as the spectrum is a Cantor set of zero Lebesgue measure, the spectral type of $H_\omega$ is necessarily purely singular. Based on known results, it is then natural to ask whether it is purely singular continuous, that is, whether the spectral measures lack point masses. In order to show that the spectral type is purely continuous, one must exclude eigenvalues for the operators $H_\omega$, which is in general a delicate endeavor. We discuss results related to the absence of point spectrum in Section~\ref{subsec:subshifteigs}.

Let us highlight one of the results from that section. In the setting of ergodic operators, one often looks for results excluding point spectrum for a.e.\ realization with respect to an ergodic measure. However, one is sometimes able to exclude eigenvalues \emph{uniformly} (that is, for every $\omega \in \Omega$, not just a.e.\ $\omega$). In general, it is somewhat rare to have a model in which one can prove uniform absence of eigenvalues.  We expand the list of known examples in Section~\ref{subsec:subshifteigs} to include periodic perturbations of simple Toeplitz subshifts for which the period is commensurate with the coding sequence and the sampling function only depends on a single entry of $x \in \X$ (see Section~\ref{subsec:subshifteigs} for definitions of Toeplitz subshifts and coding sequences).  To formulate the next result, we use the following definition: given a locally constant function $f_1:\X \to \R$ as in \eqref{eq:intro:locConstfdef}, we call $g$ the \emph{window function} and $k \in \N$ the \emph{window size}.  

\begin{theorem} \label{t:toeplitz:main}
Let $(\mathbb{X},S) \subseteq \{a,b \}^{\Z}$ be a simple Toeplitz subshift over a binary alphabet, let $f_1 \colon \mathbb{X} \to \R$ be locally constant with window size $1$, and suppose $V_\per$ is periodic with period $p$. If $(\Z_p,+1)$ is a factor of $(\mathbb{X},S)$, then, for every $x \in \mathbb{X}$, the operator $H_x + V_\per$ has no eigenvalues.
\end{theorem}

We emphasize that, in the generality formulated here, Theorem~\ref{t:toeplitz:main} is new even in the case $p=1$.
Since every simple Toeplitz subshift over a binary alphabet satisfies Boshernitzan's condition, the spectrum of $H_x + V_\per$ is in fact a Cantor set of Lebesgue measure zero in the situation of Theorem~\ref{t:toeplitz:main}, provided that the window function $g$ is nonconstant, and thus the spectral type is purely singular continuous in that scenario.
\medskip

Next, we consider periodic modifications of the random case. To model the random part, we choose $\X = {\mc A}^\Z$ the full shift on the alphabet $\mc A = \{1,2,\ldots,m\}$ with $m \in \N$ and $\mu = \mu_0^\Z$, where $\mu_0$ is a probability measure on $\mc A$.  Without loss of generality, we assume $\mu_0(\{a\})>0$ for all $a \in \mc A$. 

We further restrict to the case in which the periodic modification has period two. For each choice of $a,b \in \mc A$, there is a natural period-two element of $\X$ which we denote by $x_{ab} = (ab)^\Z$, and which is given by $x_{2n}= a$, $x_{2n+1}= b$.

\begin{theorem} \label{t:2per+rand}
Suppose $(\X,S)$ is a full shift on an alphabet with $m$ symbols, $V_\per$ has period two, and $f_1:\X \to \R$ is locally constant of window size $1$. For $\mu$-almost every $x \in \mathbb{X}$, the spectrum of the corresponding Schr\"{o}dinger operator is given by
\begin{equation} \label{eq:2-periodic-spectrum}
\sigma(H_{x} + V_\per) = \bigcup_{a,b \in \mc A} \sigma(H_{x_{ab}} +V_\per).
\end{equation}
\end{theorem}

 As in the case of Theorems~\ref{t:bosh:main} and \ref{t:bosh:permult}, our framework can incorporate other periodic decorations. The statement of Theorem~\ref{t:2per+rand} remains true if $H_x+V_\per$ is replaced by $\Delta+\lambda_\per V_x$ with $\lambda_\per$ of period two (and a similar replacement for $H_{x_{ab}}$). While we were completing this manuscript, we learned that William Wood had independently proved Theorem~\ref{t:2per+rand} (and the more general statement from which it is derived)  in the case $m=2$ in addition to obtaining finer results such as an explicit calculation of spectral gaps \cite{WoodPreprint}.

There is a natural na\"{i}ve generalization of the statement in Theorem ~\ref{t:2per+rand} when $V_\per$ has period larger than two. We give an example in Section~\ref{subsec:per+bern} to show that this generalization already fails for potentials of period three. 

Theorem~\ref{t:2per+rand} is related to an interesting open question: does the almost-sure spectrum of $H_x+V_\per$ always have finitely many connected components whenever $V_\per$ is periodic? This is well-known when $V_\per$ has period one (i.e., is constant) and Theorem~\ref{t:2per+rand} gives an affirmative answer to the question when the period is two.
\medskip

Finally, we consider periodic perturbations of quasi-periodic potentials. Let $\T = \R/\Z$ denote the circle, and suppose $\alpha \in \T$ is irrational. Here, one sometimes wants to observe phenomena that depend on variations of the frequency or sampling function. So, for $f_1 \in C(\T,\R)$ and $x \in \T$, the potential $V_{f_1,\alpha, x}$ is given by
\begin{equation} \label{eq:intro:qpVxdef}
V_{f_1,\alpha,x}(n) = f_1(n\alpha + x).
\end{equation}
The corresponding Schr\"odinger operator is denoted $H_{f_1, \alpha ,x}$. It is well known (and not hard to show with minimality and strong operator convergence) that there is a compact set $\Sigma = \Sigma_{f_1,\alpha}$ with $\Sigma = \sigma(H_{f_1,\alpha,x})$ for all $x \in \T$. Quasi-periodic operators have been heavily studied over the years; we direct the reader to the survey \cite{MarxJito2017ETDS} for a guide on the literature.

In this setting, one has the following result. See Section~\ref{sec:qpper} for definitions of the $p$-step cocycle and sub/super-criticality of a cocycle.

\begin{theorem} \label{t:qppermain}
Let $\alpha \in \T$ be irrational. If $f_1:\T \to \R$ is a non-constant real-valued trigonometric polynomial and $V_\per$ is periodic, then there exist $0< \lambda_1 < \lambda_2 < \infty$ such that the $p$-step transfer matrix cocycle associated with $H_{\lambda f_1, \alpha ,x} + V_\per$ is subcritical for every energy in $\Sigma_{\lambda f_1, \alpha}$ when $|\lambda|\leq \lambda_1$ and supercritical on $\Sigma_{\lambda f_1, \alpha}$ when $|\lambda| \geq \lambda_2$.
\end{theorem}

The analysis of cocycle dynamics pays dividends for the spectral analysis. Indeed, Avila's almost-reducibility theorem \cite{AvilaARAC1, AvilaARAC2} implies that the spectral type is purely absolutely continuous in the subcritical region. On the other hand, there is a well-established road map to proving localization in the regime of positive Lyapunov exponents. Relatively straightforward modifications of Bourgain--Goldstein's argument \cite{BourgainGoldstein2000Ann} allow one to prove Anderson localization (i.e., pure point spectrum with exponentially decaying eigenfunctions) for large $|\lambda|$ and a.e.\ frequency and phase.

The idea of reorganizing a periodic decoration of a quasi-periodic potential by passing to blocks of period length has been applied fruitfully in other recent works, such as \cite{WXYZZpreprint}, which proved specific results for the quasiperiodic \emph{mosaic model}, which corresponds to a specific choice of trigonometric polynomial on the product system.

The \emph{almost-Mathieu operator} is given by choosing the sampling function $f_1(x) = 2\lambda \cos(2\pi x)$, that is,
\begin{equation}
V_{\lambda,\alpha,x}^{\rm AMO}(n) = 2\lambda \cos(2\pi(n\alpha+x)).
\end{equation}
We write $H_{\lambda,\alpha,x}^{\rm AMO}$ for the corresponding operator. This operator family has been the subject of numerous investigations in recent decades; we point the reader to \cite{MarxJito2017ETDS} for a thorough account of the history.

\begin{theorem} \label{thm:AMOper}
Suppose $V_\per:\Z \to \R$ is periodic and $\alpha$ is irrational. The $p$-step cocycle associated with $H_{\lambda,\alpha,x}^{\rm AMO} + V_\per$ is subcritical on the spectrum when $|\lambda|$ is sufficiently small and supercritical on the spectrum when $|\lambda|$ is sufficiently large.
\end{theorem}

\begin{remark}
In fact, by following the proof of Theorem~\ref{t:qppermain} closely, one sees that the $p$-step cocycle associated with the periodic perturbation of the AMO is supercritical on the spectrum whenever $|\lambda|>1$.
\end{remark}

The structure of the paper follows. We recall some general facts about product systems and the transfer matrix cocycle in Section~\ref{sec:m&uesys}. We discuss the case of product systems in which one factor is a subshift in Section~\ref{sec:persubshift} in particular giving the proofs of Theorems~\ref{t:bosh:main}, \ref{t:bosh:permult}, \ref{t:toeplitz:main}, and \ref{t:2per+rand}. We discuss the quasiperiodic case in Section~\ref{sec:qpper}, proving Theorems~\ref{t:qppermain} and \ref{thm:AMOper}.

\section*{Acknowledgements}

The authors are very grateful to Jon Chaika for numerous helpful discussions and suggestions. We also want to thank the American Institute of Mathematics for hospitality and support during a January 2022 visit, during which part of this work was completed.
P.G.\ acknowledges support by the German Research Foundation (DFG) via the Collaborative
Research Centre (CRC 1283) and would like to thank the Department of Mathematics at Rice University for its kind hospitality when parts of this work were initiated.

\section{Preliminaries} \label{sec:m&uesys}

\subsection{Minimal Systems}

We start from a topological dynamical system $(\mathbb{X},S)$ with $\mathbb{X}$ a compact metric space and $S$ a homeomorphism $\X \to \X$.

\begin{definition}
Given a dynamical system $(\X,S)$ as above, we say that $K \subseteq \X$ is an $S$-\emph{minimal component} of $\X$ if it is closed, nonempty, $S$-invariant, and minimal with respect to those properties (i.e., no proper closed subset of $K$ satisfies those properties). We say that $(\mathbb{X},S)$ is \emph{minimal} if $\X$ is an $S$-minimal component. Equivalently, $(\X,S)$ is minimal if and only if the only $S$-invariant and closed subsets of $\mathbb{X}$ are $\mathbb{X}$ and $\varnothing$.
\end{definition}

One can check that two minimal components of $(\X,S)$ must either be identical or disjoint and hence one can speak of the number of minimal components.

\begin{definition}  \label{DEF:Sfmdef}
For $m \in \N$, we denote by $\Sf(m,\X,S)$ the number of minimal components of $(\X,S^m)$. Whenever $\X$ and $S$ are clear from context, we suppress the dependence and simply write $\Sf(m)$.
\end{definition}

For a given minimal system $(\X,S)$, we will consider the properties of the function $\Sf(m)$ defined in Definition~\ref{DEF:Sfmdef}.  The following result goes back to \cite[Thm.~2.24]{GH55}; compare also \cite[Thm.~3.1]{Ye}.

\begin{fact}
\label{FACT:minimal-decomposition}
For each $m \in \N$, denoting $q = \Sf(m)$, one has $q|m$ and a disjoint decomposition $\mathbb{X} = \mathbb{X}_1 \cup \ldots \cup \mathbb{X}_q$ such that each $\mathbb{X}_j$ is $S^m$-minimal and 
\begin{equation} \label{eq:minimalDecomp}
S(\mathbb{X}_j) = \mathbb{X}_{j+1 \, \mathrm{mod} \, q} \quad \forall1 \leq j \leq q.
\end{equation}
\end{fact}

Some properties of this function are listed below, compare \cite[Rem.~3.6, Thm.~3.8]{Ye}.

\begin{fact}
\label{FACT:s(n)-properties}
The function $\Sf$ satisfies the following properties. If $m_1$ and $m_2$ are relatively prime, then $\Sf(m_1m_2) = \Sf(m_1) \Sf(m_2)$. Furthermore, for each prime $p$, there exists a number $\ell_p \in \N_0 \cup \{\infty \}$ such that $\Sf(p^{\ell}) = \min \{ p^{\ell}, p^{\ell_p} \}$ for all $\ell \in \N_0$.
\end{fact}

Note that fixing the assignment $p \to \ell_p$ for every prime number $p$ determines the function $\Sf$ uniquely.
Alternatively, the function $\Sf$ is completely characterized by the subgroup of topological eigenvalues of the Koopman operator. Let us briefly recall:

\begin{definition}
Given a topological dynamical system $(\X,S)$, we say that $z \in \C$ is a \emph{topological eigenvalue} of $(\X,S)$ if it is an eigenvalue of the induced operator $C(\X) \to C(\X)$ given by $f \mapsto f\circ T$. A function $0 \not\equiv f \in C(\X)$ for which $f\circ T \equiv z f$ is then called a \emph{continuous eigenfunction} of $(\X,S)$.
\end{definition}

It is well-known and not hard to show that every topological eigenvalue of $(\X,S)$ is unimodular and the set of all eigenvalues of $(\X,S)$ comprises a countable subgroup of the circle $\partial \D$; see, e.g., \cite[Chapter~5]{Walters}.

We will see shortly that the function $\Sf$ may be characterized by the subgroup of topological eigenvalues that are also roots of unity (compare Proposition~\ref{PROP:s(n)-eigenvalue}).

\begin{lemma}
\label{LEM:top-eigenvalue-criterion}
Let $(\X,S)$ be minimal.
For every $m \in \N$, $\me^{2 \pi \mi/m}$ is a topological eigenvalue of $(\mathbb{X},S)$ if and only if $\Sf(m) = m$.
\end{lemma}

\begin{proof}
Suppose $\Sf(m) = m$, that is, $\mathbb{X}$ decomposes into $m$ disjoint $S^m$-minimal components $\mathbb{X}_1, \ldots, \mathbb{X}_m$ satisfying \eqref{eq:minimalDecomp}. Then
\[
f_m(x) = \sum_{j=1}^{m} \me^{2 \pi \mi j/m} \chi^{}_{\mathbb{X}_j}(x)
\]
is a continuous eigenfunction of $(\X,S)$ with eigenvalue $\me^{2 \pi \mi /m}$.

On the other hand, suppose that $\me^{2 \pi i /m}$ is a topological eigenvalue with continuous eigenfunction $f_m$. Let $x \in \mathbb{X}$ be given. By the choice of $f_m$, we have $f_m(S^j x) = \me^{2 \pi \mi j/m} f_m(x) =: k_j$ for all $1 \leq j \leq m$. Since $f_m$ is continuous and $f_m \circ S^m = f_m$, the spaces $\mathbb{X}_j := f_m^{-1}(\{k_j\})$ are disjoint, closed and $S^m$-invariant.\footnote{Note that disjointness follows from $f_m(x) \neq 0$, which in turn is a consequence of minimality of $(\X,S)$ and nontriviality of $f_m$.}  This implies $\Sf(m) \geq m$, which implies $\Sf(m)=m$, since $\Sf(m)$ divides $m$.
\end{proof}

\begin{prop}
\label{PROP:s(n)-eigenvalue}
For every $m \in \N$, $\Sf(m)$ is the largest divisor $q$ of $m$ such that $\me^{2 \pi \mi /q}$ is a topological eigenvalue.
\end{prop}

\begin{proof}
Let $\Sf(m) = q$. By general properties of the function $\Sf$, it follows that $q$ divides $m$ and that $\Sf(q) = q$. The latter implies that $\me^{2 \pi \mi/q}$ is an eigenvalue because of Lemma~\ref{LEM:top-eigenvalue-criterion}. Suppose there exists a larger divisor $\ell > q$ of $m$ such that $\me^{2 \pi \mi/\ell}$ is an eigenvalue. Then, $\Sf(m) \geq \Sf(\ell) = \ell >q = \Sf(m)$, a contradiction.
\end{proof}

By definition, the system $(\mathbb{X},S^m)$ is minimal precisely if $\Sf(m) = 1$. With the help of the preceding result, we immediately obtain the following characterization.

\begin{coro}
The system $(\mathbb{X},S^m)$ is minimal for all $m\in \N$ if and only if there are no topological eigenvalues $\me^{2 \pi \mi \alpha}$ of $(\mathbb{X},S)$ with $\alpha \in \Q \setminus \Z$.
\end{coro}

A system $(\X,S)$ for which $(\X,S^m)$ is minimal for every $m \in \N$ is said to be \emph{totally minimal}.

\subsection{Uniquely Ergodic Systems}
We now turn to the case in which $(\mathbb{X},S)$ is uniquely ergodic. The discussion of $S^m$-ergodic probability measures is very similar to the discussion of $S^m$-minimal components for minimal systems $(\mathbb{X},S)$. 

\begin{definition}  \label{DEF:MSfmdef}
Given a topological dynamical system $(\X,S)$ and $m \in \N$, we denote by $\MSf(m,\X,S)$ the number of $S^m$-ergodic Borel probability measures on $\mathbb{X}$; as before, we write $\MSf(m)$ whenever $(\X,S)$ is clear from context. Given an $S$-invariant measure $\mu$, we say that $z \in \C$ is an \emph{eigenvalue} of $(\X,S,\mu)$ if it is an eigenvalue of the Koopman operator on $L^2(\X,\mu)$, which is given by $f \mapsto f\circ T$ for $f \in L^2(\X,\mu)$. A measurable function $0 \not\equiv f \in L^2(\X,\mu)$ for which $f\circ T = z f$ (in $L^2(\X,\mu)$) is then called an \emph{eigenfunction} of $(\X,S,\mu)$.
\end{definition}

The properties of this function are precisely the same as for $\Sf(m)$; see the Appendix for details. We quote here only the result that is most important for the following discussion, compare Lemma~\ref{LEM:n-eigenvalues} for the proof.

\begin{prop}
\label{PROP:n-eigenvalues-pre}
Let $(\mathbb{X},S,\mu)$ be uniquely ergodic. For every $m \in \N$, $\me^{2 \pi \mi/m}$ is an  eigenvalue of $(\mathbb{X},S,\mu)$ if and only if $\MSf(m) = m$.
\end{prop}

 Proposition~\ref{PROP:n-eigenvalues-pre} is likely well-known, but we could not find an exact reference.
 
 \subsection{Cocycles and Hyperbolicity}
 \label{subsec:cocycle} 
 Given a topological dynamical system $(\Omega,T)$, a sampling function $f \in C(\Omega,\R)$, a $T$-ergodic measure $\mu$, and Schr\"odinger operators as in \eqref{eq:intro:VomegaDef}, it is natural to study the spectral properties of $H_\omega$ via the eigenvalue equation $H_\omega \psi = z \psi$ with  $z \in \C$. One can readily see that $H_\omega \psi = z \psi$ for some $\psi \in \C^\Z$ if and only if
\begin{equation}
\begin{bmatrix} \psi_{n+1} \\ \psi_{n} \end{bmatrix} = \begin{bmatrix}
z-f(T^n\omega) & -1 \\ 1 & 0
\end{bmatrix} \begin{bmatrix} \psi_n \\ \psi_{n-1} \end{bmatrix} \quad \forall n \in \Z.
\end{equation}
Defining
\begin{equation}
A_z(\omega) = \begin{bmatrix} z - f(T\omega) & -1 \\ 1 & 0\end{bmatrix},
\end{equation}
the associated \emph{cocycle} $(T,A_z): \Omega\times \C^2 \to \Omega \times \C^2$ is given by $(T,A_z)(\omega,v) = (T\omega,A_z(\omega)v)$. The iterates $(T,A_z)^n = (T^n,A_z^n)$ of this map can then be computed for $n \in \Z$:
\begin{equation}
A_z^n(\omega) 
= \begin{cases} A_z(T^{n-1}\omega) \cdots A_z(\omega) & n \geq 1 \\
I & n=0 \\
[A_z^{-n}(T^n\omega)]^{-1} & n \leq -1. \end{cases}
\end{equation}
Naturally, one has $H_\omega \psi = z\psi$ if and only if
\begin{equation}
\begin{bmatrix} \psi_{n+1} \\ \psi_{n} \end{bmatrix} = A_z^n(\omega) \begin{bmatrix} \psi_1 \\ \psi_{0} \end{bmatrix} \quad \forall n \in \Z.
\end{equation}
The associated \emph{Lyapunov exponent} is given by 
\begin{equation} L(z) = \lim_{n \to \infty} \frac1n \int_\Omega \log\|A^n_z(\omega)\| \, d\mu(\omega). \end{equation}
We say that $(T,A_z)$ is \emph{uniformly hyperbolic} if for constants $c,\lambda >0$ one has
\begin{equation}
\|A_z^n(\omega)\| \geq c \me^{\lambda|n|} \text{ for all } n \in \Z.
\end{equation}
If $L(z)>0$ but $(T,A_z)$ is not uniformly hyperbolic, it is said to be \emph{non-uniformly hyperbolic.} In the event that $(\Omega,T)$ is a minimal dynamical system (and $f$ is continuous), there is a fixed set $\Sigma \subseteq \R$ for which $\Sigma = \sigma(H_\omega)$ for all $\omega \in \Omega$. This set is characterized dynamically by Johnson's theorem \cite{Johnson1986JDE}, which says that $\Sigma = \R \setminus U$, where $U$ denotes the set of $E \in \R$ for which $(T,A_E)$ is uniformly hyperbolic.

\section{Periodic and Subshift} \label{sec:persubshift}

The first class of product systems that we will consider will be products of subshifts and cyclic groups, which is motivated by the question of stability of results for operators defined by subshifts under periodic perturbations. On one hand, for subshifts satisfying the Boshernitzan condition, one often observes zero-measure Cantor spectrum and purely singular continuous spectral type for such operators. On the other hand, the spectra of random operators can be written as the union of the spectra of periodic realizations.  Thus, the section splits into three main subsections: in the first subsection, we  explore the stability of zero-measure spectrum under periodic perturbations;  in the second subsection, we discuss the stability of purely continuous spectrum; and in the third section, we discuss the spectra assoicated with products in which the subshift factor is a full shift.

\subsection{Zero-Measure Cantor Spectrum}

It is well-known that if $(\X,S)$ is a subshift satisfying Boshernitzan's criterion and $f_1:\X \to \R$ is locally constant, then the Schr\"odinger operators $H_x = \Delta+V_x$ with potential $V_x(n) = f_1(S^nx)$ exhibit a dichotomy. Either $V_x$ is periodic for all $x \in \X$ or $\sigma(H_x)$ is a Cantor set of zero Lebesgue measure for every $x \in \X$ \cite{DamanikLenz-Bosh}. One may naturally ask whether this holds under the addition of a periodic background, that is, given a periodic potential $V_\per$, is it true that $H_x + V_\per$  either is periodic or has zero-measure Cantor spectrum? Naturally, this leads to an investigation of $\Omega = \X \times \Z_p$ where $p$ denotes the period of $V_\per$.  As such, one is then interested in whether Boshernitzan's criterion is stable under such products. Of course, some care is needed, since it is clear that minimality and unique ergodicity need not be invariant under taking products. In that vein, our first result is a stability result for Boshernitzan's criterion under products with finite (hence periodic) subshifts.

\subsubsection{Stability of Boshernitzan's Criterion}
Let us begin with a few definitions and set up notation.
\begin{definition}
Let $\mc A$ be a finite alphabet and $\mc A^+ = \cup_{\ell>0} \mc A^{\ell}$. We denote the left shift on $\mc A^{\Z}$ by $S$.
 Given a finite word $w \in \mc A^\ell \subseteq \mc A^+$, we write $|w|=\ell$ for the length of $w$, and we fix the notation $w^{\Z}$ for the sequence
\begin{equation}
\ldots www.www \ldots,
\end{equation}
where the single dot separates the $-1$th and the $0$th position. For $u \in \mc A^+$ and $x \in \mc A^+$ or $x \in \mc A^{\Z}$, we write $u \triangleleft x$ if $u$ is a subword of $x$. Given $j,k \in \Z$ with $j \leq k$, we define $x_{[j,k]} = x_j \ldots x_k$.
If $\mathbb{X}$ is a closed, $S$-invariant subset of $\mc A^{\Z}$, we call $(\mathbb{X},S)$ a subshift. We denote by $\mc L(\mathbb{X})$ the set of legal words in $\mathbb{X}$, that is $\mc L(\mathbb{X}) = \{ u \in \mc A^{+} : u \triangleleft x \mbox{ for some } x \in \mathbb{X} \}$. The set of legal words of length $n$ is given by $\mc L_n(\mathbb{X}) = \mc L (\mathbb{X}) \cap \mc A^n$. For $u \in \mc L(\mathbb{X})$, we define the corresponding cylinder set as
\begin{equation}
[u] = \{ x \in \mathbb{X} : x_{[0,|u|-1]} = u \}.
\end{equation}
\end{definition}

In the following let $(\mathbb{X},S)$ and $(\mathbb{X}',S')$ denote subshifts over alphabets $\mc A$ and $\mc A'$, respectively.
The \emph{diagonal shift} $T = S \times S'$ on the direct product $\mathbb{X} \times \mathbb{X}'$ is defined via
$T(x,x') = (S x, S' x')$. The system $(\X\times \X',T)$ is a shift on ordered pairs of sequences and could just as readily be viewed as a shift on sequences of ordered pairs. More precisely, $(\mathbb{X}\times \mathbb{X}', T)$ is topologically conjugate to a subshift $(\mathbb{Y},T')$ over the alphabet $\mc B = \mc A \times \mc A'$ in a canonical fashion.

Let us explain this in more detail. Let $\pi_1 \colon \mc B \to \mc A$ and $\pi_2 \colon \mc B \to \mc A'$ denote the canonical projections, extended to $\mc B^{\Z}$ as morphisms. Define $\mathbb{Y} = \pi_1^{-1}(\mathbb{X}) \cap \pi_2^{-1}(\mathbb{X}') \subseteq \mc B^{\Z}$ and let $T'$ be the left shift restricted to $\Y$. The map $\varphi \colon \mathbb{Y} \to \mathbb{X} \times \mathbb{X}'$, given by $\varphi (y) = (\pi_1(y) , \pi_2(y))$ is a topological conjugation. We will move freely between these two representations. Abusing notation slightly, we identify $T$ with $T'$.

Given $\omega \in \mathbb{X}\times \mathbb{X}'$, our main object of interest is the Schr\"{o}dinger operator $H_{\omega} \colon \ell^2(\Z) \to \ell^2(\Z)$,
\begin{equation}
(H_{\omega} \psi)_n = \psi_{n+1} + \psi_{n-1} + f(T^n \omega ) \psi_n,
\end{equation}
where $f \colon \mathbb{X} \times \mathbb{X}' \to \R$ is a continuous function. 

\begin{definition} \label{def:boshper:locConstonProduct}
Suppose $(\X,S)$ is a minimal subshift and $(\X',S')$ is a periodic subshift, by which we mean a subshift that consists of all translates of a single periodic sequence. We say that $f$ is \emph{locally constant} on $\X \times \X'$ if, up to a finite shift, $f$ is of the form
\begin{equation} \label{eq:boshper:locconstonprodshiftdef}
f(x,x') = g(x_0 \cdots x_{k-1}, x')
\end{equation}
for some $k \in \N$ and a suitable function $g:\mc A^k \times \X' \to \R$.
\end{definition}

\begin{remark}
Let us point out that our definition of periodic subshift is not standard in the sense that we insist that periodic subshifts are minimal. The extension of our results to non-minimal periodic subshifts is trivial and left to the reader.
\end{remark}

 Changing from $\mathbb{X}$ to a $k$-block partition, we can assume that $g$ depends only on a single entry. Let us make this precise in the following paragraph.

For $x \in \mc A^{\Z}$, let $x^{[k]} \in (\mc A^k)^{\Z}$ denote the $k$-block partition, given by $x^{[k]}_m = x_{[m,m+k-1]}$ for all $m \in \Z$. We emphasize that $\mc A^{[k]} = \mc L_k(x)$ plays the role of the letters for this sequence. Analogously, we set $\mathbb{X}^{[k]} = \{ x^{[k]} : x \in \mathbb{X} \}$. The map 
\begin{equation} \label{eq:boshper:varphikdef}\varphi_k \colon x \mapsto x^{[k]}
\end{equation} is a topological conjugation from $(\mathbb{X},S)$ to $(\mathbb{X}^{[k]},S)$, where, abusing notation slightly, we use $S$ for both shifts. We also define by $\varphi_k$ the corresponding sliding block code on finite words $v$ with $|v| \geq k$, that is, $[\varphi_k(v)]_m = v_{[m,m+k-1]}$, for $1 \leq m \leq |v| -k +1$. In this notation,
\[
f(x,x') =  g(x^{[k]}_0,x') =: f^{\ast} (x^{[k]}, x'),
\]
for a suitable function $f^{\ast}$ on $\mathbb{X}^{[k]} \times \mathbb{X}'$.

\begin{definition}
For a shift-invariant measure $\mu$ on a subshift $(\mathbb{X},S)$, we set
\begin{equation}
\underline{\mu}(n) = \min \{ \mu[u] : u \in \mc L_{n}(\mathbb{X}) \}.
\end{equation}
for all $n \in \N$.
The system $(\mathbb{X},S)$ is said to satisfy $\Star$ if there exists a shift-invariant measure $\mu$ such that
\begin{equation} \label{eq:bosh:nmun}
\limsup_{n \to \infty} n \cdot \underline{\mu}(n) > 0.
\end{equation}
If in addition $(\mathbb{X},S)$ is minimal, we say that it fulfills the Boshernitzan condition, and we also refer to $(\mathbb{X},S)$ as a Boshernitzan subshift. The reference to $S$ may be dropped if it is clear from the context.
\end{definition}

By a classical result \cite{Bosh92}, every Boshernitzan subshift is uniquely (and hence strictly) ergodic. 

It is helpful to note that for a given Boshernitzan subshift and a given relatively dense sequence $(n_k)_{k \in \N}$ of positive integers, one can choose a subsequence of $(n_k)_{k \in \N}$ that validates \eqref{eq:bosh:nmun}.

\begin{lemma} \label{lem:Boshinterpolate}
Suppose $(\X,S)$ is a Boshernitzan subshift and $(n_k)_{k \in \N}$ is an increasing sequence of positive integers that is relatively dense in $\N$. Then,
\begin{equation}
\limsup_{k\to\infty} n_k \cdot \underline{\mu}(n_k) > 0.
\end{equation}
\end{lemma}

\begin{proof}
Let $\mu$ be the unique invariant measure on $(\X,S)$. Clearly $\underline{\mu}(n)$ is decreasing. Let $\ell_j \to\infty$ be chosen so that
\begin{equation} \label{eq:boshinterp:elljchoice}
\lim_{j\to\infty} \ell_j \cdot \underline{\mu}(\ell_j) = c >0.
\end{equation}
Since $n_k$ is relatively dense, we may choose a $C>0$ such that for every $j\geq 1$, there exists $k_j \geq 1$ for which $\ell_j-C \leq n_{k_j} \leq \ell_j$. One then has
\begin{align*}
n_{k_j} \cdot \underline{\mu}(n_{k_j})
\geq (\ell_j-C) \cdot \underline{\mu}(\ell_j)
= \left( 1 - \frac{C}{\ell_j} \right)\ell_j \cdot \underline{\mu}(\ell_j)
\end{align*}
so the result follows immediately from \eqref{eq:boshinterp:elljchoice}.
\end{proof}

We will show next that, for many purposes, we can assume without loss of generality that the window function $g$ from \eqref{eq:boshper:locconstonprodshiftdef} depends only on the first coordinate of $x \in \mathbb{X}$.

\begin{lemma}
If $(\mathbb{X},S)$ satisfies the Boshernitzan condition, then so does $(\mathbb{X}^{[k]},S)$ for every $k \in \N$. Moreover, one has 
\begin{equation}
\Sf(n,\X,S) = \Sf(n,\X^{[k]},S)
\end{equation}
for all $n,k \in \N$ {\rm(}cf.\ Def.~\ref{DEF:Sfmdef}{\rm)}.
\end{lemma}

\begin{proof}
Let $\mu$ denote the unique ergodic measure on $(\mathbb{X},S)$. Since $\varphi_k$ is a topological conjugation, the minimality of $(\mathbb{X}^{[k]},S)$ is immediate. By the same argument, for every $m \in \N$, the number of $S^m$-minimal components is the same for both subshifts. Further, the measure $\mu^{\ast} = \mu \circ \varphi_k^{-1}$ is shift-invariant on $(\mathbb{X}^{[k]},S)$. Since $\varphi_k$ is a bijection from $\mc L_{n+k-1}(\mathbb{X})$ to $\mc L_{n}(\mathbb{X}^{[k]})$ for all $n \in \N$, we have
\begin{align*}
\underline{\mu^{\ast}}(n) & = \min \{ \mu^{\ast}[v] : v \in \mc L_{n} (\mathbb{X}^{[k]}) \} = \min \{ \mu [u] : u \in \mc L_{n+k-1} (\mathbb{X}) \}
\\ & = \underline{\mu}(n+k-1),
\end{align*}
which implies that $(\X^{[k]},S)$ satisfies $\Star$.
\end{proof}

Let $(\mathbb{X},S)$ be an aperiodic Boshernitzan subshift over the  alphabet $\mc A$ and $(\mathbb{X}',S')$ a periodic subshift over the finite alphabet $\mc A'$. Let $\mu$ be the unique invariant measure on $(\mathbb{X},S)$ and let $\mu'$ be the unique invariant measure on the periodic subshift $(\mathbb{X}',S')$. Then, $\nu = \mu \times \mu'$ is a $T$-invariant measure on $\mathbb{X} \times \mathbb{X}'$.

The main result of this section reads as follows.

\begin{theorem}
\label{THM:boshernitzan+periodic}
Let $(\mathbb{X},S)$ be a Boshernitzan subshift, let $(\mathbb{X}',S')$ be a $p$-periodic subshift, and define $(\Omega,T) = (\X \times \X',S\times S')$. 
\begin{enumerate}
\item[{\rm(a)}] The number of $T$-minimal components of $\Omega$ is precisely  $\Sf(p) = \Sf(p,\X,S)$, the number of $S^p$-minimal components of $\X$ {\rm(}cf.\ Def.~\ref{DEF:Sfmdef}{\rm)}. We denote the $T$-minimal components of $\Omega$ by $\Omega_1,\ldots,\Omega_{\Sf(p)}$.
\item[{\rm(b)}] Let $f$ be locally constant on $\X \times \X'$, and consider the associated family of Schr\"odinger operators $H_\omega = \Delta+V_\omega$ with $V_\omega$ as in \eqref{eq:intro:VomegaDef}. For each $1 \le j \le \Sf(p)$, there exists a compact set $\Sigma_j$ such that $\sigma(H_\omega) = \Sigma_j$ for every $\omega \in \Omega_j$. In particular, there are no more than $\Sf(p)$ distinct sets that may arise as spectra corresponding to operators $H_\omega$ with $\omega \in \Omega$.\footnote{Later on, we will abbreviate this observation by writing $\#\{\sigma(H_\omega) : \omega \in \Omega \} \leq \Sf(p)$.}
\item[{\rm(c)}] For every $1 \le j \le \Sf(p)$, one has the following dichotomy. Either $V_\omega$ is periodic for all $\omega \in \Omega_j$ or $\Sigma_j$ is a Cantor set of zero Lebesgue measure.
\end{enumerate}
\end{theorem}

\begin{remark}
One can have examples in which $V_\omega$ is periodic for $\omega$ in one minimal component but not every minimal component. Consequently, this provides examples for which $\Sigma_j$ is a Cantor set and $\Sigma_k$ is a finite-gap set for some $j$ and for some $k \neq j$; see Remark~\ref{REM:Toeplitz-different-spectra}.
\end{remark}

 As a consequence of our work in the present section, we deduce the following characterization, which may be of independent interest:

\begin{theorem} \label{t:bosh+per:prodminchar}
Let $(\mathbb{X},S)$ be a Boshernitzan subshift, let $(\mathbb{X}',S')$ be a $p$-periodic subshift, and let $\Sf(p) = \Sf(p,\X,S)$ be the number of $S^p$-minimal components of $\X$ {\rm(}cf.\ Def.~\ref{DEF:Sfmdef}{\rm)}. The following are equivalent:
\begin{enumerate}
\item[{\rm(a)}] The system $(\mathbb{X}\times \mathbb{X}',T)$ is {\rm(}topologically conjugate to{\rm)} a Boshernitzan subshift
\item[{\rm(b)}] $(\X \times \X',T)$ is minimal.
\item[{\rm(c)}] $\Sf(p)=1$.
\item[{\rm(d)}] $(\mathbb{X},S)$ has no eigenvalues of the form $\me^{2 \pi \mi k/p}$  aside from the trivial eigenvalue $1$.
\end{enumerate}
\end{theorem}

In the following, let $p$ denote the period of points in $\mathbb{X}'$. Then, $(\mathbb{X} \times \mathbb{X}',T)$ is topologically conjugate to a discrete suspension of $(\mathbb{X}, S^p)$ with constant height $p$, which explains why we focus on properties of $(\mathbb{X},S^n)$ for $n \in \N$. Let us expand a bit on this connection.

\begin{lemma}
\label{LEM:minimal-comp-of-product}
Let $(\X,S)$ be a Boshernitzan subshift and $\X'$ be a periodic subshift of period $p$. The action of $T$ on $\mathbb{X}\times \mathbb{X}'$ decomposes into minimal components. The number of $T$-minimal components in $\mathbb{X} \times \mathbb{X}'$ coincides with the number of $S^p$-minimal components in $\mathbb{X}$.
\end{lemma}

\begin{proof}
Let $m = \Sf(p,\X,S)$ be the number of $S^p$-minimal components in $\mathbb{X}$. We define an equivalence relation on $\mathbb{X}$ by $x \sim y$ if and only if $x$ and $y$ are in the same $S^p$-minimal component. Equivalently, $x \sim y$ whenever $y$ is in the $S^p$-orbit closure of $x$.  
For $(x,x'),(y,y') \in \mathbb{X} \times \mathbb{X}'$, let us denote $(x,x') \leadsto (y,y')$ if $(y,y')$ is in the $T$-orbit closure of $(x,x')$. Fixing $x^{\ast} \in \mathbb{X}'$, every point in $\mathbb{X} \times \mathbb{X}'$ is of the form $T^j(x, x^{\ast})$ for some $0 \leq j \leq p-1$ and $x \in \mathbb{X}$. We obtain
\begin{align*}
T^j(x,x^{\ast}) \leadsto T^k(y,x^{\ast}) 
& \iff (x,x^{\ast}) \leadsto (y,x^{\ast}) \\ & \iff x \sim y 
\\ & \iff T^k(y,x^{\ast}) \leadsto T^j(x,x^{\ast}),
\end{align*}
showing that $\leadsto$ is an equivalence relation on $\mathbb{X}\times \mathbb{X}'$. Further, the calculation above reveals that the equivalence classes satisfy $[T^j(x,x^{\ast})] = [(x,x^{\ast})]$ for all $0 \leq j \leq p-1$ and that they are in one-to-one correspondence to the equivalence classes in $\mathbb{X}$. By definition of $\leadsto$, the equivalence classes $[(x,x^{\ast})]$ are $T$-minimal components.
\end{proof}

Our next goal is to show that $(\mathbb{X},S^m)$ satisfies the Boshernitzan condition on each minimal component. To this end, we interpret $(\mathbb{X},S^m)$ as a subshift over the alphabet $\widetilde{\mc A} = \mc L_m(\mathbb{X})$. The legal words of length $k$ in this alphabet are given by $\widetilde{\mc L}_k = \widetilde{\mc A}^k \cap \mc L(\mathbb{X})= \mc L_{mk}(\mathbb{X})$. That is, there is a topological conjugacy from $(\X,S^m)$ to a subshift $(\widetilde{\X},S)$ over the alphabet $\widetilde{\mc A}$ defined by 
\[\X \ni x \mapsto \widetilde{\varphi}(x) \in \widetilde{\mc A}^\Z,\]
where $(\widetilde{\varphi}(x))_k = x_{[km,(k+1)m-1]}$.

\begin{lemma} \label{lem:bosh+per:minCompsSatisfyBosh}
Let $(\X,S)$ be a Boshernitzan subshift and let $m \in \N$. Each $S^m$-minimal component of $\mathbb{X}$ satisfies the Boshernitzan condition.
\end{lemma}

\begin{proof}
Let $q = \Sf(m,\X,S)$ and write $\mathbb{X} = \mathbb{X}_1 \cup \ldots \cup \mathbb{X}_q$ as a disjoint decomposition of $\mathbb{X}$ into $S^m$-minimal components, as provided by Fact~\ref{FACT:minimal-decomposition}. For $1 \leq j \leq q$, the rescaled restriction $\mu_j = q \mu|_{\mathbb{X}_j}$ is an $S^m$-invariant probability measure on $\mathbb{X}_j$. Since the minimal components are clopen sets, there is a minimal distance separating them. Hence, there exists a $k_0 \in \N$ such that for all $k \geq k_0$, the set $\widetilde{\mc L}_k$ splits into the disjoint union $\widetilde{\mc L}_k = \widetilde{\mc L}_k(\mathbb{X}_1) \cup \ldots \cup \widetilde{\mc L}_k(\mathbb{X}_q)$. That is, for $u \in \widetilde{\mc L}_k$, the set $[u]$ is completely contained in one of the $S^m$-minimal components. For such $k$, we find
\begin{align}
\nonumber
\underline{\mu_j}(k) 
&= \min \{ \mu_j[u] : u \in \widetilde{\mc L}_k(\mathbb{X}_j) \} \\
\nonumber
& = q \min \{ \mu[u] : u \in \widetilde{\mc L}_k(\mathbb{X}_j) \}
\\ 
\nonumber
&\geq q \min \{ \mu[u] : u \in \mc L_{mk}(\mathbb{X}) \} \\
\label{eq:bosh:mujLBcomp}
& = q \, \underline{\mu}(mk).
\end{align}
Since $(\mathbb{X},S)$ is assumed to be a Boshernitzan subshift, Lemma~\ref{lem:Boshinterpolate} gives
\[
\limsup_{k\to\infty} k\cdot \underline{\mu}(mk) >0,
\]
Combining this with \eqref{eq:bosh:mujLBcomp} gives us
\[
\limsup_{k \to \infty} k \cdot \underline{\mu_j}(k) 
\geq \limsup_{k\to\infty} qk\cdot \underline{\mu}(mk) >0,
\]
which concludes the argument.\end{proof}

Interpreting $(\mathbb{Y},T)$ as a (discrete) suspension of $(\mathbb{X},S^m)$, we find that an analogous statement holds for this system. An alternative proof is given below.

\begin{coro} \label{coro:bosh+per:minCompsofYSatisfyBosh}
Each $T$-minimal component in $(\mathbb{Y},T)$ satisfies the Boshernitzan condition.
\end{coro}

\begin{proof}
Let $u \in \mc L(\mathbb{Y})$ be a legal word of length $k \in \N$ and denote by $u_1 = \pi_1(u)$ and $u_2 = \pi_2(u)$ the projections to $\mc L(\mathbb{X})$ and $\mc L(\mathbb{X}')$ respectively. Hence, $[u] = \varphi^{-1}([u_1] \times [u_2])$ and therefore, since $\nu = \mu\times \mu'$, $(\nu \circ \varphi) [u] = \mu[u_1] \mu'[u_2] = \mu[u_1]/p$ for $k$ large enough that $[u_2]$ is a singleton in $\X'$.  Again for large enough $k$, the cylinder set $[u]$ is contained in precisely one $T$-minimal component and property $(\ast)$ is inherited from $(\mathbb{X},S,\mu)$.
\end{proof}

We are now in a position to prove some of our main results.

\begin{proof}[Proof of Theorem~\ref{THM:boshernitzan+periodic}]  By Lemma~\ref{LEM:minimal-comp-of-product}, $\X \times \X'$ has $\Sf(p)$ distinct $T$-minimal components, which proves (a). By minimality and continuity of $f$, $\sigma(H_\omega)$ is constant on each minimal component of $T$, and thus there are at most $\Sf(p)$ sets that can arise as $\sigma(H_\omega)$ for $\omega \in \X \times \X'$, proving (b). 

By Corollary~\ref{coro:bosh+per:minCompsofYSatisfyBosh}, each minimal component of $(\X \times \X',T)$ is topologically conjugate to a Boshernitzan subshift. By our assumptions on $f$, the restriction of $f$ to each minimal component is locally constant. In view of these observations, the conclusion of part (c) follows from \cite[Theorem~2]{DamanikLenz-Bosh}. \end{proof}

\begin{proof}[Proof of Theorem~\ref{t:bosh:main}] With notation as in the statement of the theorem, let $\X'$ denote the set of translates of $V_\per$, which is clearly a $p$-periodic subshift. The result follows from Theorem~\ref{THM:boshernitzan+periodic} by choosing the sampling function $f(x,x') = f_1(x)+x'(0)$. If $V_x+V_\per$ is periodic, so is $V_x$ and hence every translate of $V_x$ is also periodic. Thus, the dichotomy for the minimal components in Theorem~\ref{THM:boshernitzan+periodic} yields the claimed dichotomy for $x \in \X$.
\end{proof}

\begin{proof}[Proof of Theorem~\ref{t:bosh:permult}]
Similar to Theorem~\ref{t:bosh:main}, this follows from Theorem~\ref{THM:boshernitzan+periodic} by letting $\X'$ denote the set of translates of $\lambda_\per$ and choosing sampling functions of the form $f(x,x') = x'(0) \cdot f_1(x)$.
\end{proof}

\begin{remark} \label{rem:boshper:lambdaperXperiodDoubling}
Unlike in Theorem~\ref{t:bosh:main}, the (a)periodicity of $\lambda_\per V_x$ may depend on $x$. This is most easily seen for the case in which $\X$ is the period-doubling subshift, which is generated by the substitution $\vartheta \colon a \mapsto ab, b \mapsto aa$ over a binary alphabet $\mc A = \{a,b\} \subseteq \R$. Define $f_1$ to be evaluation at the origin and
\[\lambda_\per(n) = \begin{cases} 1 & n \text{ is even,} \\ 0 & n \text{ is odd.}
\end{cases} \]
The reader can check that the pointwise product $\lambda_\per V_x$ is periodic for some, but not all $x\in \X$.

Let us recall that the subshift $\X_\vartheta$ associated to a substitution $\vartheta:\mc A \to \mc A^+$ may be defined as 
\[ \X_\vartheta = \{ x \in \mc A^\Z : \forall n \in \Z, \, k \in \N, \ x_{[n,n+k-1]} \triangleleft \vartheta^m(a) \text{ for some } m \in \N, \, a \in \mc A\}, \]
that is, $\X_\vartheta$ is precisely the set of sequences whose finite subwords may be found in words of the form $\vartheta^m(a)$ with $a \in \mc A$.
\end{remark}

\begin{proof}[Proof of Theorem~\ref{t:bosh+per:prodminchar}] (a) $\implies$ (b) is trivial, and (b) $\implies$ (a) is a consequence of Corollary~\ref{coro:bosh+per:minCompsofYSatisfyBosh}. One has (b) $\iff$ (c) by Lemma~\ref{LEM:minimal-comp-of-product}.

Since topological eigenvalues comprise a subgroup of the unit circle, one can check that $\me^{2\pi \mi k/p}$ is an eigenvalue if and only if $\me^{2\pi \mi/d}$ is an eigenvalue, where $d = \gcd(k,p)$. Thus, (c) $\iff$ (d) follows from Proposition~\ref{PROP:s(n)-eigenvalue}. \end{proof}

The number of distinct spectra that can arise in the present setting is bounded above by the number of minimal components of $(\Y,T')$, which can be related to topological eigenvalues of $(\X,S)$ via Lemma~\ref{LEM:minimal-comp-of-product} and Proposition~\ref{PROP:s(n)-eigenvalue}. We conclude the present subsection by showing that the requirement for the eigenvalues to be topological is automatic for Boshernitzan subshifts.

\begin{prop}
\label{PROP:minimal-ergodic}
Let $(\X,S)$ be a Boshernitzan subshift with unique invariant measure $\mu$. Given $m \in \N$, let $\mathbb{X}_1, \ldots, \mathbb{X}_q$ denote the $S^m$-minimal components of $\mathbb{X}$. The $S^m$-ergodic measures on $\mathbb{X}$ are given by the set
\begin{equation}
\{ q \, \mu|_{\mathbb{X}_j} : 1 \leq j \leq q \}.
\end{equation}
In particular, $\Sf(m) = \MSf(m)$ and $(\mathbb{X}, S^m)$ is minimal if and only if it is uniquely ergodic.
\end{prop}

\begin{proof}
By $S$-invariance of $\mu$, we have $\mu(\mathbb{X}_j) = 1/q$ for all $1 \leq j \leq q$. Further, $S^m(\mathbb{X}_j) = \mathbb{X}_j$ for all $1 \leq j \leq q$ because $q$ divides $m$. Hence, each of the measures
\[
\mu_j = q \, \mu |_{\mathbb{X}_j}
\]
is an $S^m$-invariant probability measure on $\mathbb{X}$.
Since $\X_j$ and $\X_k$ are disjoint for $j\neq k$, one has $\mu_j \perp \mu_k$ for $j \neq k$.
Let $\varrho$ be an arbitrary, $S^m$-invariant measure on $\mathbb{X}$. By Lemma~\ref{lem:bosh+per:minCompsSatisfyBosh}, the system $(\mathbb{X}_j, S^m|_{\mathbb{X}_j}, \mu_j)$ fulfills the Boshernitzan condition and is therefore uniquely ergodic for all $1 \leq j \leq q$. Hence, the restriction of $\varrho$ to $\mathbb{X}_j$ is a multiple of $\mu_j$ for all $1\leq j \leq q$. This implies that the set $\{ \mu_j : 1 \leq j \leq q \}$ coincides with the set of extremal points in the space of $S^m$-invariant probability measures on $\mathbb{X}$ and hence with the set of $S^m$-ergodic measures thereupon.
\end{proof}

The following consequence of Proposition~\ref{PROP:minimal-ergodic} is of interest in its own right.

\begin{coro}
\label{COR:top=meas-eigenvalues}
Let $(\mathbb{X},S)$ be a Boshernitzan subshift with unique invariant measure $\mu$. Every eigenvalue $\me^{2 \pi \mi \alpha}$ of $(\X,S,\mu)$ with $\alpha \in \Q$ is a topological eigenvalue of $(\X,S)$.
\end{coro}

\begin{proof}
Without loss of generality, we assume $\alpha >0$ and write $\alpha = m/n$ with $m,n \in \N$ and $m,n$ relatively prime. Then, $\me^{2\pi \mi \alpha}$ is a (topological) eigenvalue precisely if $\me^{2 \pi \mi/n}$ is a (topological) eigenvalue. Hence, we may assume $\alpha = 1/n$. If $\me^{2 \pi \mi /n}$ is an eigenvalue, then $\MSf(n) = n$ due to Lemma~\ref{LEM:n-eigenvalues}. Since the number of $S^n$-minimal components and $S^n$-ergodic measures on $\mathbb{X}$ is the same by Proposition~\ref{PROP:minimal-ergodic}, we find $\Sf(n) = \MSf(n) = n$. With Lemma~\ref{LEM:top-eigenvalue-criterion} we conclude that $\me^{2 \pi \mi /n}$ is a topological eigenvalue.
\end{proof}

Let us emphasize that this result does \emph{not} extend to irrational eigenvalues. In \cite[Sec.~6]{BDM05} the authors construct an explicit example of a linearly recurrent subshift that has non-topological (irrational) eigenvalues.

\subsubsection{Reflection Symmetries}
In general, points in different minimal components $(\mathbb{X} \times \mathbb{X}',T)$ can give rise to the same spectrum. We illustrate this in the case that $\mathbb{X}$ and $\mathbb{Y}$ have some reflection symmetries.
We define the reflection operator $\Refl_k$ at position $k \in \halfZ$ via
\begin{equation}
\Refl_k(x)_j = x_{2k -j} \mbox{ for all } j \in \Z,
\end{equation}
with a similar definition for $\Refl_k$ on $\Omega =\X \times \X'$.
By a short calculation we obtain $S^{\ell} \circ \Refl_k = \Refl_{k - \ell/2} = \Refl_k \circ S^{-\ell}$ for all $\ell \in \Z$, $k \in \frac{1}{2}\Z$. Let us say that $f:\Omega \to \R$ is \emph{reflective} if it is locally constant and the window function $g$ can be chosen to satisfy $g(w^{\Refl}) = g(w)$ where $(w_1\cdots w_\ell)^{\Refl} = w_\ell w_{\ell-1} \cdots w_1$ denotes the reflection of the word $w$.

\begin{lemma}
\label{LEM:spectral-symmetry}
For every $\omega \in \mathbb{X} \times \mathbb{X}'$ and $k \in \halfZ$, we have $\sigma(H_{\Refl_k( \omega)}) = \sigma(H_{\omega})$ for any reflective sampling function. Indeed, $H_\omega$ and $H_{\Refl_k(\omega)}$ are unitarily equivalent.
\end{lemma}

\begin{proof}
Suppose $f$ is reflective, and write
\[ f(\omega)= g(\omega_m \cdots \omega_{m+\ell-1}) \]
for a window function $g:\mc A^\ell \to \R$ satisfying $g(w^{\Refl}) \equiv g(w)$. Define $k' = k-c$ where $c = m + \frac{\ell-1}{2} \in \halfZ$ denotes the center of the window. The reader can check that the operator $U_{k'}:\ell^2(\Z) \to \ell^2(\Z)$ given by $[U_{k'}\psi](n)= \psi(2k'-n)$ is a unitary involution satisfying $U_{k'} H_\omega U_{k'} = H_{\Refl_k(\omega)}$. Indeed, one has $U_{k'} \Delta U_{k'} = \Delta$ and $U_{k'} V_\omega U_{k'} = \Refl_{k'} V_\omega$. The assumption on $f$ and the choice of $k'$ yields $V_{\Refl_k(\omega)} = \Refl_{k'} V_\omega$. Indeed, by our choice of $k'$ and our assumption on $g$, we have
\begin{align*}
[\Refl_{k'} V_\omega](n)
= f(T^{2k'-n} \omega)
& = g(\omega|_{[2k'-n+m,2k'-n+m+\ell-1]}) \\
& = g(\omega|_{[2k-n-m-\ell+1,2k-n-m]}) \\
& = g(\omega|_{[2k-n-m-\ell+1,2k-n-m]}^{\Refl}) \\
& = f(T^n \Refl_k \omega)\\
& = V_{\Refl_k(\omega)}(n) 
\end{align*}
 which concludes the proof.\end{proof}

\begin{example} \label{ex:boshzm:tmDiffCompsSameSpec}
Assume that $\mathbb{X}$ is the Thue--Morse subshift arising from the primitive substitution $\vartheta \colon a \mapsto ab, b \mapsto ba$ and that $\mathbb{X}'$ is $2$-periodic. Since $\Sf(2) = 2$ (compare Example~\ref{ex:boshzm:constLengthSubSfunction}), the system $(\mathbb{X} \times \mathbb{X}',T)$ has two minimal components. For an arbitrary $x' \in \mathbb{X}'$, we have $\Refl_0(x') = x'$. Consider
\begin{align}
\label{eq:boshzm:tmseqdef1}
x^{\ast} 
& = \lim_{n \to \infty} \ldots \vartheta^{2n}(a) \vartheta^{2n}(a). \vartheta^{2n}(a) \vartheta^{2n}(a) \ldots \\
\label{eq:boshzm:tmseqdef2}
& = \ldots abbabaabbaababba . abbabaabbaababba \ldots, 
\end{align}
which is a fixed point under $\vartheta^2$. Since both $\vartheta^2(a)$ and $\vartheta^2(b)$ are reflection symmetric, the same holds for $\vartheta^{2n}(a)$ for all $n \in \N$. Hence, $\Refl_{-1/2}(x^{\ast}) = x^{\ast}$ (recall that the dot in \eqref{eq:boshzm:tmseqdef1} separates positions $-1$ and $0$) and therefore $\Refl_0(x^{\ast}) = S^{-1} x^{\ast}$. Because $(x^{\ast},x')$ and $\Refl_0((x^{\ast},x')) = (S^{-1}x^{\ast},x')$ belong to different $T$-minimal components, Lemma~\ref{LEM:spectral-symmetry} implies that $\sigma(H_\omega)$ is independent of $\omega$ for any choice of reflective sampling function.
\end{example}

In the following, we investigate this phenomenon in a more systematic fashion. 
\begin{definition}
For a minimal subshift $(\mathbb{X},S)$, the following are equivalent
\begin{enumerate}
\item There are $k \in \halfZ$, $x \in \mathbb{X}$ such that $\Refl_k(x) \in \mathbb{X}$.
\item $\Refl_k(\mathbb{X}) = \mathbb{X}$ for all $k \in \halfZ$.
\end{enumerate} When these statements hold, we say that $\X$ is \emph{reflection symmetric}. 
\end{definition}
Suppose now that both $\mathbb{X}$ and $\mathbb{X}'$ are reflection symmetric. Because $\mathbb{X}'$ consists of a single finite orbit, this implies that $\Refl_k(x')$ coincides with a shift of $x'$ for every $k \in \halfZ$. Consequently, there exists a point $x' \in \mathbb{X}'$ and $k \in \{0,1/2\}$ such that $\Refl_k(x') = x'$. Let $p$ be the period of $\mathbb{X}'$.

\begin{lemma}
\label{LEM:symm-classes}
Suppose $(\X,S)$ is a minimal subshift, $p \in \N$, $m = \Sf(p)$, and let $\X_1,\ldots,\X_m$ denote the $S^p$-minimal components as in Fact~\ref{FACT:minimal-decomposition}. The reflection operator $\Refl_k$ acts as a reflection on the tuple $(\mathbb{X}_1, \ldots, \mathbb{X}_m)$ in the sense that there exists $\ell \in \halfZ \cap [1,\frac{m+1}{2}]$ such that
\begin{equation}
\Refl_k (\mathbb{X}_j) = \mathbb{X}_{2\ell -j \, \mathrm{mod} \, m} \quad \text{ for all } 1 \le j \le m.
\end{equation}
\end{lemma}

\begin{proof}
For notational convenience let us define $\mathbb{X}_i := \mathbb{X}_{i \, \mathrm{mod} \, m}$ for all $i \in \Z$.
Let $x \in \mathbb{X}_1$ and let $r \in \{1,\ldots,m \}$ be such that $\Refl_k(x) \in \mathbb{X}_r$. Define $\ell = (1+r)/2$. For $1 \leq j \leq m$ and $y \in \mathbb{X}_j$, there exists a sequence $(n_i)_{i \in \N}$ such that $S^{n_i m + j -1} x \to y$ as $i \to \infty$ (because $S^{j-1} x \in \mathbb{X}_j$ and $\mathbb{X}_j$ is $S^m$-minimal). Applying $\Refl_k$, we obtain
\[
S^{-n_i m -j +1} \Refl_k(x) \to \Refl_k(y),
\]
as $i \to \infty$. This implies that $\Refl_k(y)$ is in the same $S^m$-minimal component as $S^{-j+1} \Refl_k(x) \in S^{-j+1}(\mathbb{X}_r) = \mathbb{X}_{r-j+1}$. Since $r = 2\ell-1$, we have $\Refl_k(y) \in \mathbb{X}_{2\ell - j}$. Hence, $\Refl_k(\mathbb{X}_j) \subseteq \mathbb{X}_{2\ell -j}$. By the same argument $\Refl_k(\mathbb{X}_{2\ell -j}) \subseteq \mathbb{X}_j$ and the claim follows.
\end{proof}

This symmetry relation between the minimal components reduces the upper bound for the number of different spectra roughly by a factor $1/2$. A careful case distinction on the parity of $m$ and on whether $\ell \in \Z$ or $\ell \in \Z + 1/2$ yields the following as a corollary of Lemma~\ref{LEM:spectral-symmetry} and Lemma~\ref{LEM:symm-classes}.

\begin{prop}
In the setting of Theorem~\ref{THM:boshernitzan+periodic} assume that both $\mathbb{X}$ and $\mathbb{X}'$ are reflection symmetric. Let $k \in \{0,1/2\}$ be such that $\Refl_k(x') = x'$ for some $x' \in \mathbb{X}'$. Then,
\[
\# \{ \sigma(H_{\omega}) : \omega \in \mathbb{X} \times \mathbb{X}' \} \leq
\begin{cases}
(\Sf(p) + 1)/2 & \mbox{if } \Sf(p) \in 2 \N -1,
\\ \Sf(p) / 2 + \delta(k,p) & \mbox{if } \Sf(p) \in 2 \N,
\end{cases}
\]
where $\delta(k,p) = 1$ if there exists an $x \in \mathbb{X}$ such that $\Refl_k(x)$ is in the same $S^p$-minimal component as $x$ and $\delta(k,p) = 0$ otherwise.
\end{prop}

\begin{proof}
The spectrum is constant on every $T$-minimal component of $\mathbb{X} \times \mathbb{X}'$. These are of the form $\mathbb{X}_j \times \{ x'\}$, with $1\leqslant j \leqslant m$ and $m = s(p)$. Since we have assumed $R_k(x') = x'$, it follows by Lemma~\ref{LEM:symm-classes} that there is $\ell \in \frac{1}{2} \Z \cap [1, \frac{m+1}{2} ]$ with
\[
R_k(\mathbb{X}_j \times \{x'\}) = \mathbb{X}_{2\ell -j \,\mathrm{mod}\, m} \times \{x'\},
\]
for all $1\leqslant j \leqslant m$. By Lemma~\ref{LEM:spectral-symmetry}, the spectrum is the same on $\mathbb{X}_j \times \{x'\}$ and $\mathbb{X}_{2 \ell - j \, \mathrm{mod} \, m} \times \{ x'\}$. Hence, the number of different spectra is bounded from above by the number of orbits of the map
\[
r_{\ell} \colon j \mapsto 2 \ell - j \, \mathrm{mod} \, m,
\]
on the cyclic group with $m$ elements. Each of these orbits has either one or two elements. The number $j$ is a fixed point of $r_{\ell}$ precisely if 
\[
j = \ell\; \mathrm{mod} \, m \quad \mbox{or} \quad 
j = \ell + \frac{m}{2} \; \mathrm{mod} \, m,
\]
If $m$ is odd, this condition has precisely one solution, that is, there is precisely one fixed point of $r_{\ell}$ and hence the number of $r_{\ell}$-orbits is given by $(m + 1)/2$. If $m$ is even, there are either two fixed points or no fixed point, depending on whether $\ell$ is an integer or not. The former is the case if and only if there is $x \in \mathbb{X}$ such that $R_k(x)$ and $x$ belong to the same $S^p$-minimal component. In this case, the number of distinct $r_{\ell}$-orbits is given by $(m+2)/2$. Otherwise, every orbit of $r_{\ell}$ has precisely $2$ elements, such that the number of orbits is given by $m/2$. 
\end{proof}

\subsubsection{Examples and Applications}

We consider the function $\Sf(p)$ for several prominent classes of Boshernitzan subshifts $(\mathbb{X},S)$.

\begin{example}
Assume that $(\mathbb{X},S)$ is totally ergodic. Then, $\Sf(p) = 1$ for all $p \in \N$, so $(\mathbb{X} \times \mathbb{X}',T)$ is a Boshernitzan subshift for all periods $p$. One prominent example of a totally ergodic system  is the Fibonacci subshift. This can be seen from the fact that it can be coded by an irrational rotation on the circle. More generally, every Sturmian subshift is totally ergodic by the same argument \cite{Kur03}. Sturmian subshifts will be discussed in more detail in Section~\ref{SUBSEC:Sturmian}. Finally, one also has $\Sf(p)=1$ for all $p \in \N$ if $(\X,S,\mu)$ is weak mixing.
\end{example}

\begin{example} \label{ex:boshzm:constLengthSubSfunction}
Assume that $(\mathbb{X},S)$ is the subshift associated to a primitive substitution $\vartheta$ of constant length $\ell \geq 2$ on the alphabet $\mc A$. That is, $|\vartheta(a)| = \ell$ for all $a \in \mc A$. Let $\ell_1,\cdots,\ell_r$ be the prime factors of $\ell$. Recall that we assume $(\mathbb{X},S)$ to be aperiodic. The discrete dynamical spectrum of $(\mathbb{X},S)$ was completely characterized in \cite{Dekking}. There is a number $1 \leq h \leq \# \mc A$, coprime to $\ell$, with the following property. For every $p \in \N$, we have
\[
\Sf(p) = \ell_1^{j_1} \cdots \ell_r^{j_r} h^{\min\{1,k\}},
\]
where $j_1, \ldots, j_r \in \N_0$ and $k \in \N_0$ are maximal with the property that $\ell_1^{j_1} \cdots \ell_r^{j_r} h^k | p$.
An algorithm to determine $h$ was given in \cite[Rem.~9]{Dekking}. If $\mc A$ is a binary alphabet, we have $h=1$. In this case, $(\mathbb{X} \times \mathbb{X}',T)$ is a Boshernitzan subshift precisely if $\ell$ and $p$ are coprime. In particular, this applies to the Thue--Morse substitution and the period-doubling substitution whenever $p$ is odd.
\end{example}

For general primitive substitution subshifts, characterizing the group of eigenvalues is more subtle. Following the seminal paper by Host \cite{Host86}, several characterizations of eigenvalues and criteria for special cases have been proposed. We present a small selection. A general algebraic characterization of rational eigenvalues was given in \cite[Prop.~2]{FMN96}.

Before we continue, let us introduce some notation.
\begin{definition}\label{def:boshper:submatdef}
 Given a substitution $\vartheta$, let $M$ denote the corresponding substitution matrix, that is, $M_{ab} = |\vartheta(b)|_a$ for all $a,b \in \mc A$, where $|w|_a$ denotes the number of occurences of $a$ in the word $w$.
 \end{definition}
 Given a primitive substitution, let $\lambda_1$ denote the Perron-Frobenius eigenvalue of $M$ (i.e., $\lambda_1 > 1$ is real and strictly the largest eigenvalue in absolute value) and $\lambda_2$ the second largest eigenvalue in absolute value. The following was shown in \cite[Thm.~1.2]{KSS05}.

\begin{prop}
Suppose that $|\lambda_2|>1$. Then, $(\mathbb{X},S)$ is topologically mixing if and only if $\gcd(\{|\vartheta^n(a)| : a \in \mc A\}) =1$ for all $n \in \N$.
\end{prop}

Since topological mixing implies weak mixing, we find that in this case $(\mathbb{X} \times \mathbb{X}',T)$ is a Boshernitzan subshift, irrespective of $p$.
For a recent result, characterizing the rational eigenvalues in the case of proper primitive substitutions, compare also \cite[Lem.~10]{DG19}.

\begin{example}
An interval exchange transformation (IET) on $\X = [0,1)$ is defined by a choice of a permutation $\pi$ on $\{1,2,\ldots,n\}$ and $\lambda = (\lambda_1,\ldots,\lambda_n) \in \R_+^n$ such that $\lambda_1+\cdots + \lambda_n =1$. Given such a $\pi$ and $\lambda$, the associated IET, $S = S_{\pi,\lambda}$, acts on $\X$ by partitioning $\X$ into $n$ intervals where the $j$th interval has length $\lambda_j$ and then rearranging those intervals according to the permutation $\pi$. More precisely, defining
\begin{align*}
c_k & = \sum_{j=1}^{k-1} \lambda_j, \ 1 \le k \le n+1 \\
\hat{c}_k & = \sum_{j=1}^{\pi(k)-1} \lambda_{\pi^{-1}(j)}, \ 1 \le k \le n,
\end{align*}
(where empty sums vanish by convention), one puts
\begin{equation}
Sx = x-c_k+\hat{c}_k \quad \text{ for } c_k\leq x < c_{k+1}, \ 1 \le k \le n.
\end{equation}
One says that the permutation $\pi$ is \emph{irreducible} if there is no $1 \le k<n$ for which $\pi(\{1,2,\ldots,k\}) = \{1,2,\ldots,k\}$. Without loss of generality, one restricts attention to irreducible $\pi$. In this case, Veech showed that $S_{\pi,\lambda}$ is totally ergodic for (Lebesgue) a.e.\ $\lambda$ \cite{Veech1984AJM}. In particular, almost every IET satisfies $\Sf(n) = 1$ for all $n \in \N$. Furthermore, almost every IET satisfies Boshernitzan's condition (in the sense that the subshift associated to the natural coding of $(\X,S_{\pi,\lambda})$ satisfies (B)) \cite{Bosh85DMJ}. In particular, the results of the present paper apply to Schr\"odinger operators defined by almost every IET  with periodic background of any period.
\end{example}

An interesting class of subshifts for which the function $\Sf(p)$ can be made explicit is the class of \emph{Toeplitz subshifts}. These are particular symbolic extensions of \emph{odometers}. We give a brief sketch of this connection and refer to \cite{Downarowicz} for a comprehensive review on this topic. An odometer is defined via a \emph{scale} $t = (t_n)_{n \in \N}$ of natural numbers such that $t_n$ divides $t_{n+1}$ for all $n \in \N$. The corresponding odometer is given by the inverse limit
\begin{equation}
\Z(t) = \left\{ (m_n)_{n \in \N} \in \prod_{n = 1}^\infty \Z_{t_n} : m_n = m_{n+1} \, \mathrm{mod} \, t_n \mbox{ for all } n \in \N \right\}.
\end{equation}
Equipped with the normalized Haar measure $\nu$ on the topological group $\Z(t)$ and the map $\tau \colon \Z(t) \to \Z(t)$,
\[
(\tau m)_n = m_n + 1 \, \mathrm{mod} \, t_n,
\]
for all $n \in \N$, the system $(\Z(t),\tau,\nu)$ is strictly ergodic. The \emph{multiplicity function} of an odometer $\Z(t)$ assigns to each prime number $p$ a multiplicity $\kappa(p) \in \N \cup \{ \infty \}$ which is the supremum over all $k$ such that there is $n \in \N$ with $p^k$ dividing $t_n$. Two odometers are isomorphic precisely if they have the same multiplicity function.  
Another useful characterization of the multiplicity function is that $\kappa(p)$ is the supremum over all $k \in \N$ such that $\me^{2\pi\im/p^k}$ is an eigenvalue of $(\Z(t),\tau,\nu)$. Note that for odometers, all eigenvalues are topological eigenvalues; compare the discussion in Remark~\ref{REM:frequency-module}.
Hence, due to Proposition~\ref{PROP:s(n)-eigenvalue} the value of the multiplicity function $\kappa(p)$ coincides with the multiplicity $\ell_p$, alluded to in Fact~\ref{FACT:s(n)-properties}. That is,
\begin{equation} \label{eq:boshper:odometerSfct}
\Sf(p^\ell) = \min\{p^{\ell},p^{\kappa(p)} \},
\end{equation}
for every prime number $p$.
A subshift $(\mathbb{X},S)$ is called a (topological) \emph{extension} of the odometer $(\Z(t),\tau)$ if there exists a factor map $\pi \colon \mathbb{X} \to \Z(t)$, that is, a surjective and continuous map satisfying $\pi \circ S = \tau \circ \pi$ on $\mathbb{X}$. Such an extension is called \emph{almost $1$--$1$} if there is a dense set of points $z \in \Z(t)$ such that $\pi^{-1}(z)$ is a singleton.

\begin{definition}
\label{DEF:toeplitz}
We call a subshift $(\mathbb{X},S)$ a Toeplitz subshift if it is minimal and an almost $1$--$1$ extension of an odometer $(\Z(t),\tau)$.
\end{definition}

In this case, the odometer $(\Z(t),\tau)$ is the maximal equicontinuous factor of $(\mathbb{X},S)$; compare \cite[Ch.~1]{PF02} for details on this notion. In particular, both systems share the same group of topological eigenvalues and hence the same function $\Sf(n)$. 
Not every Toeplitz subshift is a Boshernitzan subshift \cite{LiuQu}. In the next section, we therefore restrict our attention to the more special class of \emph{simple Toeplitz subshifts} on a binary alphabet $\mc A$, as these are known to satisfy Boshernitzan's condition \cite[Prop.~4.1]{LiuQu}.


\subsection{Exclusion of Eigenvalues} \label{subsec:subshifteigs}

At present, it is unclear to us whether the addition of a periodic potential can alter the property of permitting Schr\"odinger eigenvalues. We therefore check which of the known criteria for excluding eigenvalues are stable under periodic perturbations.

\subsubsection{Uniform Absence of Eigenvalues: Simple Toeplitz Subshifts}

We borrow some notation from \cite{LiuQu}.
Let $\mc A \subseteq \R$ be a binary alphabet and $s = (b_k,n_k)_{k \in \N}$ a coding sequence, with $b_k \in \mc A$ and $n_k \geq 2$ for all $k \in \N$. 
Recursively, we define $\widetilde{w}_1 = b_1$ and
\begin{equation}
\label{EQ:w-tilde-recursion}
\widetilde{w}_{k+1} = \widetilde{w}_k^{n_k} b_k^{-1} b_{k+1},
\end{equation}
for all $k \in \N$. Since $\widetilde{w}_k$ is a prefix of $\widetilde{w}_{k+1}$ for all $k \in \N$, there is a well-defined limit
\begin{equation}
x(s) := \lim_{k \to \infty} \widetilde{w}_k
\end{equation}
in $\mc A^\N$. The minimal subshift
\begin{equation}
\mathbb{X}(s) = \{ x \in \mc A^{\Z} : x_{[j,k]} \triangleleft x(s) \text{ for all } j\leq k \}
\end{equation}
is called the simple Toeplitz subshift corresponding to $s$. 
Every point in $\mathbb{X}(s)$ is called a simple Toeplitz sequence. 
We make the non-triviality assumption that $(b_k)_{k \in \N}$ is not eventually constant. This ensures that $\mathbb{X}(s)$ is non-periodic. In fact, up to a modification of the sequence $(n_k)_{k \in \N}$, there is no loss of generality in assuming that the sequence $(b_k)_{k \in \N}$ is alternating. That is, we assume $b_k \neq b_{k+1}$ for all $k \in \N$.

\begin{remark}
\label{REM:s(n)-toeplitz}
The subshift $(\mathbb{X}(s),S)$ is indeed a Toeplitz subshift as defined in Definition~\ref{DEF:toeplitz}.
In fact, it is an almost $1$--$1$ extension of the odometer $(\Z(t),\tau)$, with scale $t = (t_k)_{k \in \N}$, given by $t_k= \prod_{j=1}^k n_j$, for all $k \in \N$. Hence, for every prime $p$ and $\ell \in \N_0$,
\[
\Sf(p^\ell) = \min\{p^{\ell}, p^{\kappa(p)} \},
\]
where $\kappa(p)$ is the total number of times (counted with multiplicities) that $p$ appears as a factor of $n_k$, as we vary $k \in \N$.
\end{remark}

Let us write $\widetilde{v}_k$ for the word that emerges from $\widetilde{w}_k$ by exchanging the last letter $b_k$ with the unique letter $b_k' \in \mc A \setminus\{b_k\}$. By construction, every $\omega \in \mathbb{X}(s)$ can be written as a concatenation of the words $\widetilde{w}_k$ and $\widetilde{v}_k$, for every level $k$, where two occurrences of $\widetilde{v}_k$ are separated by at least $n_k -1$ occurrences of $\widetilde{w}_k$. In fact, by \eqref{EQ:w-tilde-recursion} and the assumption that $(b_k)_{k \in \N}$ is alternating, we have
\begin{equation}
\label{EQ:w-v-tilde-recursions}
\widetilde{w}_{k+1} = \widetilde{w}_k^{n_k - 1} \widetilde{v}_k
\quad \mbox{and} \quad
\widetilde{v}_{k+1} = \widetilde{w}_k^{n_k}.
\end{equation}

We combine this structure with a $p$-periodic background potential, where we assume that $p$ is in some sense commensurate with the Toeplitz structure.
\begin{definition}
We call a number $p \in \N$ \emph{commensurate} with a coding sequence $s = (a_k,n_k)_{k \in \N}$ if there is a number $k_0 \in \N$ such that $p$ divides $t_{k_0} = \prod_{k= 1}^{k_0} n_k$.
\end{definition}

This is the case precisely if $\Sf(p) = p$, compare Remark~\ref{REM:s(n)-toeplitz}.
In the following, let $\Omega(s,p) = \mathbb{X}(s) \times \Z_p$, equipped with the map $T \colon (x,m) \to (Sx, m+1)$. Naturally, arithmetic in the second coordinate is performed modulo $p$.
Within this section, we assume that the sampling function $f$ is locally constant of window size one. Shifting if necessary, we may  assume without loss of generality that $f$ is of the form
\begin{equation} \label{eq:boshper:toeplitzSfAssmp}
f(x,m) = g(x_0,m)
\end{equation}
for some function $g:\mc A \times \Z_p \to \R$. Thus, for all $(x,m) \in \Omega(s,p)$ and $n \in \N$,
\[
V_{(x,m)}(n) = g(x_n , m + n).
\]
In the following, it will be convenient to regard $\Omega(s,p)$ as a subshift over the alphabet $\mc A' = \mc A \times \Z_p$, where we define for $\omega = (x,m) \in \Omega(s,p)$, with some abuse of notation
\[
\omega_n = (x_n, m+n) \in \mc A',
\]
for all $n \in \Z$. With this convention, $V_{\omega}(n) = g(\omega_n)$.

The following theorem is the main result of this section.

\begin{theorem}
\label{THM:Toep-p-no-eigenvalues}
Let $\mathbb{X}(s)$ be a simple Toeplitz subshift over a binary alphabet $\mc A$, assume that $p \in \N$ is commensurate with $s$, and suppose that $f$ is of the form \eqref{eq:boshper:toeplitzSfAssmp}. Then, for all $\omega \in \Omega(s,p)$, the Schr\"odinger operator $H_\omega$ has no eigenvalues.
\end{theorem}

To the best of our knowledge, this result is new in the stated generality, even if the periodic background is dropped. We therefore explicitly spell out the case $p = 1$ as a corollary.

\begin{coro}
\label{COR:no-eigenvalues-simple-T}
Let $\mathbb{X}(s)$ be a simple Toeplitz subshift over a binary alphabet $\mc A \subseteq \R$. Then, for all $x \in \mathbb{X}(s)$, the Schr\"odinger operator $H_x$ with potential $V_x(n) = x_n$ has no eigenvalues.
\end{coro}

Corollary~\ref{COR:no-eigenvalues-simple-T} was known in the case $n_k \equiv 2$, which corresponds to the period-doubling subshift \cite{D01}. On the other hand, if $n_k \geq 4$ for all $k \in \N$, uniform absence of eigenvalues follows from \cite[Thm.~1.3]{LiuQu}. Hence, Corollary~\ref{COR:no-eigenvalues-simple-T} builds a bridge between those cases of simple Toeplitz subshifts on a binary alphabet where uniform absence of eigenvalues is already known to hold.

Another consequence of Theorem~\ref{THM:Toep-p-no-eigenvalues} is that periodic decorations of sequences in the period-doubling subshift cannot produce Schr\"odinger eigenvalues if the period is a power of $2$. Again, we assume that $f$ is of the form specified in \eqref{eq:boshper:toeplitzSfAssmp}. 

\begin{coro}
Let $\vartheta \colon a \mapsto ab, b \mapsto aa$ be the period-doubling substitution and $(\mathbb{X}_\vartheta,S)$ the corresponding subshift. Assume that $p = 2^n$ for some $n \in \N_0$ and $\Omega(\vartheta,p) = \mathbb{X}_{\vartheta} \times \Z_p$. Then, for every $\omega \in \Omega(\vartheta,p)$ the Schr\"odinger operator $H_{\omega}$ has no eigenvalues.
\end{coro}

\begin{proof}
It is straightforward to verify that $\Omega(\vartheta,p) = \Omega(s,p)$, with coding sequence $s = (b_k,2)_{k \in \N}$, where $b_{2k - 1} = a$ and $b_{2k} = b$ for all $k \in \N$. Indeed, an induction argument shows that the pair $\widetilde{w}_{k+1} = \vartheta^{k}(a)$ and $\widetilde{v}_{k+1} = \vartheta^k(b)$ satisfies the defining relation in \eqref{EQ:w-v-tilde-recursions} for all $k \in \N_0$.
\end{proof}

Recall that $\Sf(p)=p$ whenever $p$ is commensurate with $s$, and therefore $(\Omega(s,p),T)$ decomposes into precisely $p$ minimal components, each given by the $T$-orbit closure $\Omega(s,p)_m$ of $(x^{\ast},m)$, for $m \in \Z_p$ and some fixed $x^{\ast} \in \mathbb{X}(s)$. For the sake of definiteness, let us choose $x^{\ast}$ such that it coincides with $x(s)$ on $\N$. 
Without loss of generality, we restrict our attention to $(\Omega(s,p)_0,T)$.
The structure of points in $\Omega(s,p)_0$ is inherited from the original Toeplitz structure. More precisely, the point $(x^{\ast},m)$ starts with the word $w_k$, given by
\[
(w_k)_m = ((\widetilde{w}_k)_m,m),
\]
for all $k \in \N$ and $1 \leq m \leq |w_k|$. This corresponds to a $p$-periodic decoration of the word $w_k$. Since $p$ is assumed to be commensurate with the coding sequence $s$, there exists a number $k_0 \in \N$ such that $p$ divides $|w_k|$ for all $k \geq k_0$. For such $k$, the relation
\begin{equation}
\label{EQ:w-recursion}
w_{k+1} = w_k^{n_k} a_k^{-1} a_{k+1}
\end{equation}
is inherited from \eqref{EQ:w-tilde-recursion}, where $a_k = (b_k,p)$  for all $k \geq k_0$. Similarly, we obtain
\begin{equation}
\label{EQ:w-v-recursions}
w_{k+1} = w_k^{n_k - 1} v_k
\quad \mbox{and} \quad
v_{k+1} = w_k^{n_k},
\end{equation}
for all $k \geq k_0$ from \eqref{EQ:w-v-tilde-recursions}, where $v_k$ emerges from $w_k$ by exchanging the last letter $a_k = (b_k,p)$ with $a'_k = (b'_k,p)$.

\begin{remark}
\label{REM:Toeplitz-different-spectra}
Note that $V_{\omega}$ is periodic for all $\omega \in \Omega(s,p)_0$ precisely if $g(b_k,p) = g(b_k',p)$, in which case $\sigma(H_{\omega})$ is a union of intervals. More generally, $V_{\omega}$ is periodic for $\omega \in \Omega(s,p)_m$ precisely if $g(b_k,m) = g(b_k',m)$. Note that it is possible to choose $g$ such that this property holds for some, but not all $m \in \Z_p$. In this case, $\sigma(H_{\omega})$ is a Cantor spectrum of Lebesgue measure $0$ for some, but not all $\omega \in \Omega(s,p)$. However, this effect cannot occur if $g$ is of the form $g(b,m) = g_1(b) + g_2(m)$.
\end{remark}

A central tool in the study of spectral properties of $H_\omega$ is the \emph{trace map}. For $E \in \R$ and $a \in \mc A'$, we define
\begin{equation} \label{eq:boshper:MEadef}
M_E(a) = \begin{bmatrix}
E - g(a) & -1
\\ 1 & 0
\end{bmatrix}
\end{equation}
and for a word $w = a_1 \cdots a_n$, let
\begin{equation} \label{eq:boshper:MEwdef}
M_E(w) = M_E(a_n) \cdots M_E(a_1).
\end{equation}Let us define
\begin{equation}
x_k = x_k(E) = \tr M_E(w_k), \quad y_k = y_k(E) = \tr M_E(v_k).
\end{equation}
This is related to the cocycle notation introduced in Section~\ref{subsec:cocycle} via
\[M_E(w_k) = A_E^{|w_k|}(\omega), \quad k \in \N,\]
where $\omega = (x^*,0)$.

We denote by $(S_n)_{n \in \N_0}$ the sequence of Chebychev polynomials, given by
\[
S_0(x) \equiv 0, \quad S_1(x) \equiv 1, \quad S_{n+1}(x) = x S_n(x) - S_{n-1}(x). 
\] 

\begin{lemma}
\label{LEM:trace-relations}
For every $k \geq k_0$, we have
\begin{align}
\label{eq:LEM:trace-relations:xk}
x_{k+1} & = S_{n_k}(x_k) y_k - 2 S_{n_k - 1}(x_k),\\
\label{eq:LEM:trace-relations:yk}
y_{k+1} & = S_{n_k}(x_k) x_k - 2 S_{n_k - 1}(x_k).
\end{align}
In particular,
\begin{equation}
\label{EQ:difference-recursion}
|x_{k+1} - y_{k+1}| = |S_{n_k}(x_n)| |x_k - y_k|.
\end{equation}
\end{lemma}

\begin{proof}
For all $n \in \N$, and $A \in \SL(2,\R)$ we have the relation
\[
A^n = S_n(\tr A) A - S_{n-1}(\tr A) I
\]
by Cayley--Hamilton and induction, where $I$ is the identity matrix. Using \eqref{EQ:w-v-recursions}, this yields
\[
y_{k+1} = \tr M_E(v_{k+1}) = \tr M_E(w_k^{n_k}) = S_{n_k}(x_k) x_k - 2 S_{n_k - 1}(x_k),
\]
proving \eqref{eq:LEM:trace-relations:yk}.

Recall that $w_k$ ends in $a_k$. Let $a_k'$ be the last letter of $v_k$. Then,
\begin{align*}
M_E(w_{k+1}) 
& =M_E(w_k^{n_k - 1} v_k)\\ 
& = M_E(a_k') M_E(a_k)^{-1} M_E(w_k)^{n_k}\\ 
&= S_{n_k}(x_k) M_E(v_k) - S_{n_k - 1}(x_k) M_E(a_k') M_E(a_k)^{-1},
\end{align*}
By a direct calculation, $\tr M_E(a_k') M_E(a_k)^{-1} = 2$, and thus (using \eqref{EQ:w-v-recursions} again) we have
\[
x_{k+1} = \tr M_E(w_{k+1})
= S_{n_k}(x_k) y_k - 2 S_{n_k - 1}(x_k),
\]
proving \eqref{eq:LEM:trace-relations:xk}.
\end{proof}

%
%

\begin{prop}
\label{PROP:trace-stability}
Assume that $x_k(E) = y_k(E)$ for some $k \geq k_0$. Then, for every sequence $\omega \in \Omega(s,p)_0$, the number $E$ is not an eigenvalue of $H_\omega$.
\end{prop}

\begin{proof}

Let $k \geq k_0$ and let $a = g(a_k), a' = g(a'_k)$, where $a_k$ denotes the final letter of $w_k$. On the level of transfer matrices, the variation in the right-most location between $w_k$ and $v_k$ is modeled via
$$
M_E(a_k') M_E(a_k)^{-1} = \begin{bmatrix} 1 & a - a' \\ 0 & 1 \end{bmatrix}.
$$

Thus, if 
$$
M_E(w_k) =
\begin{bmatrix} a_{1,1} & a_{1,2} \\ a_{2,1} & a_{2,2} \end{bmatrix},
$$
we obtain
\begin{align*}
M_E(v_k) & =
\begin{bmatrix} 1 & a - a' \\ 0 & 1 \end{bmatrix} \begin{bmatrix} a_{1,1} & a_{1,2} \\ a_{2,1} & a_{2,2} \end{bmatrix}
\\ & = \begin{bmatrix} a_{1,1} + (a - a') a_{2,1} & a_{1,2} + (a - a') a_{2,2} \\ a_{2,1} & a_{2,2} \end{bmatrix}.
\end{align*}
The two traces are therefore
$$
x_k(E) = a_{1,1} + a_{2,2}
$$
and
$$
y_k(E) = a_{1,1} + (a - a') a_{2,1} + a_{2,2}.
$$
If $a = a'$, the sequences $V_{\omega}$ with $\omega \in \Omega(s,p)_0$ are in fact periodic and hence do not permit Schr\"{o}dinger eigenvalues.
If $a \not= a'$, $x_k(E) = y_k(E)$ requires that $a_{2,1} = 0$.
In the case $a_{2,1} = 0$, the two matrices take the form
$$
M_E(w_k) = \begin{bmatrix} a_{1,1} & a_{1,2} \\ 0 & a_{2,2} \end{bmatrix}
$$
and
$$
M_E(v_k) = \begin{bmatrix} a_{1,1} & a_{1,2} + (a - a') a_{2,2} \\ 0 & a_{2,2} \end{bmatrix},
$$
and hence an arbitrary product of $n$ such matrices will have the form
$$
\begin{bmatrix} a_{1,1}^n & * \\ 0 & a_{2,2}^n \end{bmatrix}.
$$
It follows also in this case that the energy in question is not an eigenvalue for any element of the subshift: if $|a_{1,1}| = |a_{2,2}|^{-1} \not = 1$, then the cocycle $(T,A_E)$ (cf.\ Section~\ref{subsec:cocycle}) is uniformly hyperbolic and the energy is not in the spectrum by Johnson's theorem; and if $|a_{1,1}| = |a_{2,2}|^{-1} = 1$, then transfer matrices of this kind obviously do not admit any decaying solutions, and in particular no square-summable solutions.
\end{proof}

%

The following is a mild adaptation of a corresponding result on simple Toeplitz sequences in \cite{LiuQu}. 

\begin{prop}
\label{PROP:spectrum-trace}
If $E \in \sigma(H_\omega)$, it follows that $|x_k(E)| \leq 2$ for infinitely many $k \in \N$.
\end{prop}

\begin{proof}[Sketch of proof]
This is essentially \cite[Prop.~3.1]{LiuQu}. The interested reader can verify that all the arguments leading to this result rely on the fact that the recursion relation \eqref{EQ:w-recursion} remains true for large enough $k \in \N$. 
\end{proof}

For most simple Toeplitz subshifts, the conclusion in Theorem~\ref{THM:Toep-p-no-eigenvalues} can be derived from a combination of the $3$-block and the $2$-block Gordon lemma, similar to the discussion in \cite{D01}.

\begin{lemma}
\label{LEM:not-eventually-3-case}
Let $s = (a_k,n_k)$ be such that $n_k \neq 3$ for infinitely many $k \in \N$. Then, for all $\omega \in \Omega(s,p)$, the Schr\"odinger operator $H_\omega$ has no eigenvalues.
\end{lemma}

\begin{proof}
Without loss of generality, let $\omega \in \Omega(s,p)_0$. For the sake of establishing a contradiction, assume that $E \in \R$ is an eigenvalue of $H_\omega$ with eigenfunction $\psi \in \ell^2(\Z)$.
Let $k_1 \geq k_0$ be arbitrary. Due to Proposition~\ref{PROP:spectrum-trace}, there exists a $k \geq k_1$ such that $|x_k(E)| \leq 2$. Consider the decomposition of $\omega$ into words of the form $w_k$ and $v_k$. Let us use the symbol $\, \hat{} \,$ to note the location of the $0$th position. We proceed by cases. 

\emph{Case 1:} Assume that around the origin, $\omega$ is of one of the forms $w_k w_k \hat{v}_k$, $w_k w_k \hat{w}_k$, $\hat{w}_k w_k w_k$, or $\hat{w}_k w_k v_k$. In each of these cases, we can apply the $2$-block Gordon criterion to conclude that $\psi_n$ is of order $1$ for some $n \in \Z$ with $|n| \geq |w_{k_1}|$.

\emph{Case 2:} If Case~1 does not hold, then $\omega$ must have one of the following forms near the origin: $v_k w_k \hat{v}_k$, $v_k \hat{w}_k v_k$ or $v_k w_k \hat{w}_k v_k$ (recall that $v_k$'s cannot be adjacent and indeed must be separated by at least $n_k-1$ occurrences of $w_k$). In each of these cases, $\omega$ has the form $w_{k+1} \hat{w}_{k+1}$ near the origin. We can apply the $3$-block Gordon lemma and reach the same conclusion as in the last case, unless $\omega$ has the form $v_{k+1} w_{k+1} \hat{w}_{k+1} v_{k+1}$. Note that this requires $\omega$ to be of the form $w_{k+2}\hat{w}_{k+2}$ near the origin, which we can use to see $n_{k+1} = 3$. We can repeat the same reasoning until we reach a level $r > k$ with $n_r \neq 3$. For this level, the structure $v_r w_r \hat{w}_r v_r$ is not possible and we can apply the $3$-block Gordon lemma. In every case, $\psi_n$ is of order $1$ for some $n \in \Z$ with $|n| \geq |w_{k_1}|$.

Since $k_1$ was arbitrary, we reach a contradiction to the assumption that $\psi \in \ell^2(\Z)$. 
\end{proof}
In view of Lemma~\ref{LEM:not-eventually-3-case}, what remains is to handle the case in which $(n_k)_{k\in\N}$ is eventually identically 3.
We treat this remaining case by a centered version of the $2$-block Gordon lemma. The following lemma does not require the subshift setting. Thus, we work with an arbitrary sequence $\omega \in \R^\Z$, and the associated potential is simply $V_\omega(n) = \omega_n$. For this lemma, simply define $M_E$ by \eqref{eq:boshper:MEadef} and \eqref{eq:boshper:MEwdef} with $g(a) = a$ for $a \in \R$.
 
\begin{lemma}
\label{LEM:centered-Gordon}
Let $\omega \in \R^\Z$ and suppose there exists a strictly increasing sequence of natural numbers $(n_m)_{m \in \N}$ such that
\begin{equation} \label{eq:LEM:centered-Gordon:squareAroundOrigin}
\omega_{[-n_m,-1]} = \omega_{[0,n_m-1]},
\end{equation}
for all $m \in \N$.  For $E \in \R$, $m \in \N$, let $x_m(E) = \tr M_E(\omega_{[0,n_m-1]})$. If 
\begin{equation}
\sum_{m \in \N} |x_m(E)|^2 = \infty,
\end{equation}
then $E$ is not an eigenvalue of $H_{\omega}$.
\end{lemma}

\begin{proof}
For the sake of establishing a contradiction, assume that $\psi \in \ell^2(\Z)$ is an eigenfunction of $H_{\omega}$ for the eigenvalue $E$. Let $\Psi(n) = (\psi_n,\psi_{n-1})^\top$ for all $n \in \Z$, and $\mathcal{M}_m = M_E(\omega_{[0,n_m-1]})$ for all $m \in \N$. Note that \[\mathcal{M}_m^{-1} \Psi(n_m) = \Psi(0) = \mathcal{M}_m \Psi(-n_m) \text{ for all } m \in \N
\]
by \eqref{eq:LEM:centered-Gordon:squareAroundOrigin}. Then, due to the Caley--Hamilton theorem,
\[
\Psi(n_m) + \Psi(-n_m) = x_m(E) \Psi(0).
\]
Normalizing $\psi$ according to $\|\Psi(0)\|_2 = 1$, and using $2|a|^2 + 2|b|^2 \geq |a+b|^2$, this yields
\begin{align}
\label{EQ:trace-normalization}
|\psi_{-n_m}|^2 + |\psi_{-n_m+1}|^2 + |\psi_{n_m}|^2 + |\psi_{n_m+1}|^2 
& \geq   \frac{1}{2} |x_m(E)|^2. 
\end{align}
Without loss of generality, we can restrict $(n_m)_{m \in \N}$ to a subsequence such that $n_{m+1} > n_m + 1$ for all $m \in \N$. We then obtain by \eqref{EQ:trace-normalization},
\[
2 \|\psi\|^2 \geq \sum_{m \in \N} |x_m(E)|^2 = \infty,
\]
in contradiction to $\psi \in \ell^2(\Z)$.
\end{proof}

\begin{lemma}
\label{LEM:eventually-3-case}
Let $s = (a_k,n_k)$ be such that $n_k = 3$ for all $k$ larger than some $k_1 \geq k_0$. Then, for all $\omega \in \Omega(s,p)$, the Schr\"odinger operator $H_{\omega}$ has no eigenvalues.
\end{lemma}

\begin{proof}
Again, we restrict our attention to the case $\omega \in \Omega(s,p)_0$ without loss of generality. Assume $E \in \sigma(H_{\omega})$ is an eigenvalue with corresponding eigenfunction $\psi \in \ell^2(\Z)$. If $\omega$ is \emph{not} of the form $v_k w_k \hat{w}_k v_k$ for all $k$ larger than $k_1$, we can argue as in the proof of Lemma~\ref{LEM:not-eventually-3-case} and reach a contradiction.
Hence, the assumptions imply that $\omega$ is of the form $v_k w_k \hat{w}_k v_k$ around the origin for all $k \geq k_1$. This implies
\[
\omega_{[|w_k| - 1, -1]} = \omega_{[0,|w_k| -1]},
\]
and the trace of the corresponding transfer matrix is given by $x_k(E)$.  
By Proposition~\ref{PROP:trace-stability} and the assumption that $E$ is an eigenvalue, we have $x_k \neq y_k$ for all $k \geq k_1$. By Lemma~\ref{LEM:centered-Gordon}, it suffices to show that 
\begin{equation}
\label{EQ:trace-non-summable}
\sum_{k \in \N} |x_k(E)|^2 = \infty
\end{equation}
in order to obtain a contradiction.
For all $k \geq k_1$, the assumption $n_k = 3$ implies that the trace map is of the form
\begin{align*}
x_{k+1} & = (x_k^2 - 1) y_k - 2 x_k,
\\y_{k+1} & = (x_k^2 - 1) x_k - 2 x_k.
\end{align*}
Again, we directly obtain \eqref{EQ:trace-non-summable} unless $x_k \to 0$, which we assume in the following. This yields
\[
\lim_{k \to \infty} y_k = \lim_{k \to \infty} \frac{x_{k+1} + 2 x_k}{x_k^2 - 1} = 0.
\]
Let $k_2 \geq k_1$ be such that $|x_k| < 1/2$ for all $k \geq k_2$. For each such $k$, iterating \eqref{EQ:difference-recursion} yields
\[
|x_{k+1} - y_{k+1}| = |x_{k_2} - y_{k_2}| \prod_{m = k_2}^{k} |x_m^2 - 1|.
\]
By $x_k,y_k \to 0$ we find that $\log|x_{k+1} - y_{k+1}| \to - \infty$ as $k \to \infty$ and thereby,
\[
\lim_{k \to \infty} \sum_{m = k_2}^k \log |1 - x_m^2| = - \infty. 
\]
It is straightforward to verify that $\log|1 - x_m^2| > - c x_m^2$ for some $c>0$, due to the requirement $|x_m| < 1/2$. Therefore,
\[
\sum_{m = k_2}^{\infty} |x_m|^2 \geq - \frac{1}{c} \sum_{m = k_2}^{\infty} \log |1 - x_m^2| = \infty
\]
and the contradiction follows.
\end{proof}

\begin{proof}[Proof of Theorem~\ref{THM:Toep-p-no-eigenvalues}]
This follows from Lemma~\ref{LEM:not-eventually-3-case} and Lemma~\ref{LEM:eventually-3-case}.
\end{proof}

\begin{proof}[Proof of Theorem~\ref{t:toeplitz:main}]
This follows from Theorem~\ref{THM:Toep-p-no-eigenvalues} by choosing a sampling function of the form $f(x,m) = g(x_n,m)$ where $g:\mc A \times \Z_p \to \R$ is of the form $g(b,m) = g_1(b) + g_2(m)$.
\end{proof}
\subsubsection{Almost-Sure Exclusion of Eigenvalues}
In this section, we use Gordon's criterion in order to exclude Schr\"odinger eigenvalues almost surely for systems that exhibit a sufficient degree of local repetitions. As always, we work with locally constant sampling functions $f$.

\begin{definition}
Let $(\mathbb{X},S,\mu)$ be an ergodic subshift. For $n \in \N$, let
\[
\mathbb{X}(n) = \{ x \in \mathbb{X} : x_{[-n-1]} = x_{[0,n-1]} = x_{[n,2n-1]} \}.
\]
We say that $(\mathbb{X},S,\mu)$ satisfies the \emph{Gordon condition} if $\limsup_{n \to \infty} \mu(\mathbb{X}(n)) > 0$.
\end{definition}

By general results, every subshift that satisfies the Gordon condition exhibits almost sure absence of eigenvalues for the associated Schr\"{o}dinger operators.

\begin{prop}
If $(\X,S,\mu)$ satisfies the Gordon condition, then for any locally constant sampling function $f_1:\X \to \R$, the operator $H_x = \Delta+V_x$ has purely continuous spectrum for $\mu$-a.e.\ $x \in \X$ {\rm(}with $V_x$ defined in \eqref{eq:intro:Vxdef}{\rm)}.
\end{prop}

\begin{proof}
Given a locally constant function $f_1$, denote by $\mathbb{X}_c(f_1)$ the set of $x \in \mathbb{X}$ for which $H_x$ has purely continuous spectrum, and write
\[ \mathbb{X}_{\rm g} := \limsup_{n\to\infty} \mathbb{X}(n)  = \bigcap_{n \geq 1} \bigcup_{k \geq n} \mathbb{X}(k).\]
If $\mathbb{X}$ satisfies the Gordon condition, one has $\mu(\mathbb{X}_{\rm g}) > 0$. One also has $\mathbb{X}_{\rm g} \subseteq \mathbb{X}_{\rm c}(f_1)$. This is an immediate consequence of the three-block version of the Gordon lemma when $f_1$ has window size one \cite{D00Gord, Gordon76}. When the window size of $f$ is larger than $1$, the potentials $V_x$ may not exactly satisfy the three-block Gordon condition. However, modifying the potentials in a neighborhood of the boundary of the three-block structure can repair this at the cost of a fixed multiplicative constant on the size of the transfer matrices. Since $\X_c(f)$ is shift-invariant, it follows that $\mathbb{X}_{\rm c}(f_1)$ has full $\mu$-measure.
\end{proof}

Recall that we denote by $p$ the period of the subshift $(\mathbb{X}',S')$. In the following, let $\rho$ be a fixed but arbitrary ergodic measure on $(\mathbb{X} \times \mathbb{X}',T)$.

\begin{prop}
\label{PROP:pn-limsup}
Suppose $(\mathbb{X},S,\mu)$ satisfies $\limsup_{n \to \infty} \mu (\mathbb{X}(pn)) > 0$.
Then, the operator $H_{\omega}$ has no eigenvalues for $\rho$-almost every $\omega \in \mathbb{X} \times \mathbb{X}'$.
\end{prop}

\begin{proof}
Let $m = \Sf(p)$ and let $\X_1,\ldots,\X_m$ denote the $S^p$-minimal components as described in Fact~\ref{FACT:minimal-decomposition}. Fix $x' \in \mathbb{X}'$ and define $\nu = \mu \times \mu'$. Every ergodic measure $\rho$ on $(\mathbb{X} \times \mathbb{X}',T)$ is of the form $\rho = m \, \nu|_{\mathbb{Y}_j}$, where
\[
\mathbb{Y}_j = \bigcup_{k=0}^{p -1} T^k (\mathbb{X}_j \times \{x'\}),
\]
for $1 \leq j \leq m$, is one of the $T$-minimal components of $\mathbb{X} \times \mathbb{X}'$. Viewing $\mathbb{Y}_j$ as a subshift over the alphabet $\mc A \times \mc A'$, and using that $S(\mathbb{X}_s) = \mathbb{X}_{s+1}$ (with indices modulo $m$), we find
\[
\mathbb{Y}_j(pn) = \bigcup_{k=0}^{p-1} \mathbb{X}_{j+k}(pn) \times \{(S')^k x'\}
\]
because the $pn$-periodic block structure of $(S')^k x'$ is automatic. Note that $\mathbb{X}(pn) = \cup_{k=1}^{m} \mathbb{X}_k(pn)$ as a disjoint union because $\mathbb{X}$ can be partitioned into $\mathbb{X}_1,\ldots, \mathbb{X}_m$.
Hence, $\mu(\mathbb{X}(pn)) = \sum_{k=1}^m \mu (\mathbb{X}_k(pn))$.
Using $\mu'(\{(S')^k x'\}) = 1/p$, we obtain
\[
\rho(\mathbb{Y}_j(pn)) 
= m \sum_{k=0}^{p-1} \frac{1}{p} \mu(\mathbb{X}_{j+k}(pn)) 
= \frac{m}{p} \frac{p}{m} \sum_{k=1}^{m}  \mu(\mathbb{X}_{j+k}(pn))
=\mu(\mathbb{X}(pn)),
\]
which yields 
$
\limsup_{n \to \infty} \rho (\mathbb{Y}_j(pn)) =  \limsup_{n \to \infty} \mu(\mathbb{X}(pn)) > 0
$
and hence for $\rho$-almost every $\omega \in \mathbb{X} \times \mathbb{X}'$ the associated Schr\"{o}dinger operator $H_{\omega}$ has no eigenvalues by the Gordon criterion.
\end{proof}

One can show that the Gordon criterion holds for substitution subshifts whenever words generated by the substitution have suitable repetitions. To make this precise, we define the \emph{index} of a word $u \in \mc L(\mathbb{X})$ by
\[
\ind(u) = \sup \left\{ r\in \tfrac{1}{|u|}\Z_+ : u^r \in \mc L(\mathbb{X}) \right\}.
\]
Recall that rational powers of a word $u$ are defined as follows. For $r = n+\frac{\ell}{|u|}$ with $n \in \Z_+$, $0 \le \ell < |u|$, one defines
\[ u^r = u^{n}u_1\ldots u_\ell. \]

For substitution subshifts we obtain the following criterion, which is an adaptation of \cite[Thm.~3]{DL06b}.

\begin{prop} \label{prop:perbosh:substitutionAScontspec}
Let $(\mathbb{X},S)$ be a substitution subshift generated by a primitive substitution $\vartheta$. Assume that there exists a word $u \in \mc L(\mathbb{X})$ with $\ind(u) > 3$ and such that $p$ divides $|\vartheta^n(u)|$ for infinitely many $n$. Then $H_{\omega}$ has no eigenvalues for $\rho$-almost every $\omega \in \mathbb{X} \times \mathbb{X}'$.
\end{prop}

\begin{proof}
Since the claim is obvious if $(\mathbb{X},S)$ is periodic, we can assume that $\vartheta$ is aperiodic.
Let $(n_k)_{k \in \N}$ be an increasing sequence of natural numbers such that $p$ divides $|\vartheta^{n_k}(u)|$ for all $k \in \N$. By Proposition~\ref{PROP:pn-limsup} it suffices to show that $\limsup_{k \to \infty} \mu (\mathbb{X}(|\vartheta^{n_k}(u)|)) > 0$. 
Let $u_0$ be the first letter of $u$ and $k \in \N$. Since $uuuu_0 \in \mc L(\mathbb{X})$, the same holds for $\vartheta^{n_k}(u) \vartheta^{n_k}(u) \vartheta^{n_k}(u) \vartheta^{n_k}(u_0)$. This word contains precisely $|\vartheta^{n_k}(u_0)| +1$ blocks of the form $www$, where $|w| = |\vartheta^{n_k}(u)|$.
By the unique ergodicity of $(\mathbb{X},S)$, it suffices to bound the frequency of (non-overlapping) appearances of such a pattern in an arbitrary element $x \in \mathbb{X}$. A lower bound for this is given by the frequency of times that $(S^j x)_{j \in \Z}$ enters the set $\vartheta^{n_k}([uuuu_0])$. Denoting by $\lambda>1$ the Perron--Frobenius eigenvalue of the substitution matrix, this is given by
\[
\mu(\vartheta^{n_k}([uuuu_0])) = \frac{1}{\lambda^{n_k}} \mu([u u u u_0]),
\] 
which follows because $\vartheta$ is aperiodic and primitive \cite[Theorem~5.10 and Corollary~5.11]{queffelec2010}. Since each $j \in \Z$ such that $S^j x \in \vartheta^{n_k}([uuu u_0])$ contributes $|\vartheta^{n_k}(u_0)| + 1$ occurrences of a three block structure with the required length, we get
\[
\mu (\mathbb{X}(|\vartheta^{n_k}(u)|)) \geq |\vartheta^{n_k}(u_0)| \, \mu (\vartheta^{n_k}([uuuu_0])) \geq \frac{|\vartheta^{n_k}(u_0)|}{\lambda^{n_k}} \mu([uuuu_0]).
\]
By primitivity, $\lambda^{-n_k} |\vartheta^{n_k}(u_0)|$ converges to the corresponding entry $L_{u_0} > 0$ of the left Perron--Frobenius eigenvector $L$ as $n_k \to \infty$ and the assertion follows.
\end{proof}

\begin{example}
Assume that $\vartheta$ is a substitution of constant length $\ell$ and that $p$ is a divisor of $\ell^m$ for some $m \in \N$. If the index of the associated subshift $(\mathbb{X},S)$, given by $\ind(\mathbb{X}) = \sup \{ \ind(u) : u \in \mc L(\mathbb{X} \}$, is greater than $3$, Proposition~\ref{prop:perbosh:substitutionAScontspec} yields almost sure absence of Schr\"{o}dinger eigenvalues. This applies in particular to the case that $\vartheta$ is given by the period-doubling substitution $\vartheta \colon a \mapsto ab, \, b \mapsto aa$ and $p = 2^n$ for some $n \in \N$.
\end{example}

If the substitution matrix $M$ associated to $\vartheta$ (cf.\ Def.~\ref{def:boshper:submatdef}) is invertible over the integers, we get the following consequence of Proposition~\ref{PROP:pn-limsup}.

\begin{coro}
\label{COR:index-3}
Let $(\mathbb{X},S)$ be a substitution subshift for some primitive substitution $\vartheta$ with a substitution matrix $M$ that is invertible over $\Z$. If there exists a word $u \in \mc L(\mathbb{X})$ such that $\ind(u) > 3$ and $p$ divides $|u|$ then $H_{\omega}$ has no eigenvalues for $\rho$-almost every $\omega \in \mathbb{X} \times \mathbb{X}'$.
\end{coro}

\begin{proof}
Let $\Phi(u)$ be the abelianization of $u$, that is $\Phi(u)_a = |u|_a$ for all $a \in \mc A$. The length of $u$ is then given by $|u| = \norm{\Phi(u)}_1$. By construction, $\Phi(\vartheta^n(u)) = M^n \Phi(u)$ and hence $|\vartheta^n(u)| = \norm{M^n \Phi(u)}_1$ for all $n \in \N$. Since $M$ has only integer entries, $M v$ coincides with $M v'$ modulo $p$ whenever the entries of $v$ and $v'$ coincide modulo $p$. This shows that $(M^n v)_{n \in \N_0}$ is eventually periodic modulo $p$ for every vector $v$. Since $M^{-1}$ also has integer entries by assumption, the sequence $(M^n v)_{n \in \N}$ is in fact periodic modulo $p$. Applying this to $v = \Phi(u)$ and taking the $1$-norm, we find that $(|\vartheta^n(u)|)_{n \in \N_0}$ is periodic modulo $p$. Since $|u| \equiv 0$ modulo $p$, the same holds for $|\vartheta^n(u)|$ for infinitely many $n \in \N$, and thus Proposition~\ref{prop:perbosh:substitutionAScontspec} yields the desired result.
\end{proof}

\begin{example}
Let $\vartheta$ be the Fibonacci substitution $\vartheta \colon a \mapsto ab, \, b \mapsto a$. The substitution matrix has determinant $-1$ and is hence invertible over $\Z$. The corresponding subshift $(\mathbb{X},S)$ is a Sturmian subshift. In the notation of Berstel \cite{Berst99} we have that $\vartheta^n(a) = s_n$, with directive sequence $(1,1,1,\ldots)$. By \cite[Prop.~4]{Berst99}, the index of $\vartheta^n(a)$ is larger than $3$ for large enough $n$. On the other hand, as detailed in the proof of Corollary~\ref{COR:index-3}, $(M^n \Phi(a))_{n \in \Z}$ is periodic modulo $p$ for all $p \in \N$. Since 
\[M^{-2} \Phi(a) = \begin{bmatrix}
-1 \\ 1
\end{bmatrix}, \]
 it follows that $|\vartheta^n(a)| = (1,1) M^n \Phi(a)$ is divisible by $p$ for a lattice of integers $n$. Hence $H_{\omega}$ has no eigenvalues for $\rho$-almost every $\omega \in \mathbb{X} \times \mathbb{X}'$.
\end{example}

\subsubsection{Sturmian Sequences}
\label{SUBSEC:Sturmian}

Let $\alpha \in (0,1)$ be an irrational number with continued fraction expansion
\[
\alpha = [a_1, a_2, a_3,\ldots ] := \frac{1}{a_1 + \frac{1}{a_2 +\frac{1}{a_3 + \ldots}}}.
\]
For $\theta \in [0,1)$, consider the Sturmian sequences $s(\alpha,\theta)$ and $s'(\alpha,\theta)$, defined by
\begin{align*}
s_n(\alpha,\theta) & = \chi_{[1-\alpha, 1)} (n \alpha + \theta \, \mathrm{mod} \, 1),
\\ s'_n(\alpha,\theta) & = \chi_{(1-\alpha, 1]} (n \alpha + \theta \, \mathrm{mod} \, 1),
\end{align*}
for $n \in \Z$, where $\chi_A$ denotes the characteristic function of the set $A$. The Sturmian subshift $(\mathbb{X}_{\alpha},S)$ of slope $\alpha$, with
\[
\mathbb{X}_{\alpha} =
\{ s(\alpha,\theta) : \theta \in [0,1) \}
\cup
\{ s'(\alpha,\theta) : \theta \in [0,1) \}
\]
is a strictly ergodic subshift of $(\{0,1\}^{\Z},S)$, satisfying the Boshernitzan condition \cite{DF2}. Let $\mu$ denote the unique ergodic measure.

It is worth mentioning that $\mathbb{X}_{\alpha} \setminus \{ s(\alpha,\theta) : \theta \in [0,1) \}$ is a countable set and has therefore $\mu$-measure $0$. For measure-theoretic purposes we therefore restrict our attention to the measurable subset $\mathbb{X}_{\alpha}' = \{ s(\alpha, \theta) : \theta \in [0,1) \}$ and identify $\mu$ with the restriction of $\mu$ to $\mathbb{X}_{\alpha}'$. The bijective map $f \colon [0,1) \to \mathbb{X}_{\alpha}', \, \theta \mapsto s(\alpha,\theta)$ gives an explicit parametrization and $\mu$ coincides with the pushforward of the Lebesgue measure on $[0,1)$ under this map.
In particular, every property that holds $\mu$-almost surely is fulfilled for Lebesgue almost every $\theta \in [0,1)$.

The restriction of $s(\alpha,\theta)$ to $\N$ coincides with
$
w = \lim_{n \to \infty} w_n,
$
where
\begin{equation}
\label{EQ:w_n}
w_{-1} = 1, \quad w_0 = 0, \quad w_1 = w_0^{a_1 -1} w_{-1},
\quad w_n = w_{n-1}^{a_n} w_{n-2} \mbox{ for } n \geq 2.
\end{equation}
It was shown in \cite{DL99} that for every $n \in \N$, every bi-infinite sequence $x \in \mathbb{X}_{\alpha}$ has a unique partition into words words of the form $w_{n-1}$ and $w_n$. More precisely,
\begin{equation}
\label{EQ:n-decomposition}
x = \ldots w_{n-1} w_n^{\ell_{-1}} w_{n-1} w_n^{\ell_0} w_{n-1} w_n^{\ell_1} w_{n-1} \ldots,
\end{equation}
where the origin is somewhere in the block $w_{n-1} w_n^{\ell_0}$ and $\ell_i \in \{a_{n+1}, a_{n+1} + 1\}$ for all $i \in \Z$.

\begin{prop}
\label{PROP:unbounded-continued-fraction}
Assume that $\mathbb{X}'$ has period $p$, that $\alpha = [a_1, a_2, a_3, \ldots]$ satisfies 
\begin{equation}
\limsup_{n \to \infty} a_n \geq 4p,
\end{equation}
and that $\rho$ is a $T$-ergodic measure on $\X\times \X'$. Then, $H_{\omega}$ has no eigenvalues for $\rho$-almost every $\omega \in \mathbb{X} \times \mathbb{X}'$.
\end{prop}

\begin{proof}
By Proposition~\ref{PROP:pn-limsup}, it suffices to show that 
\begin{equation} \label{eq:PROP:unbounded-continued-fraction:goal}
\limsup_{n \to \infty} \mu(\mathbb{X}_{\alpha}(pn)) > 0.
\end{equation}Let $(n_j)_{j \in \N}$ be an increasing sequence of integers such that $n_j \to \infty$ as $j \to \infty$ and $a_{n_j +1} \geq 4p$ for all $j \in \N$. The assumption on $a_{n_j+1}$ implies that the word $w_{n_j}^{4p}$ is legal and contains blocks of the form $www$ with $|w| = |w^p_{n_j}|$ at precisely $|w^p_{n_j}| +1$ positions. Given $x \in \mathbb{X}_{\alpha}$, whenever two occurences of $w_{n_j}^{4p}$ are separated by at least $|w^p_{n_j}|$, all of these blocks $www$ indeed appear at different positions within $x$. Using the structure in \eqref{EQ:n-decomposition}, we can find an increasing sequence of positions $(q_i)_{i \in \Z}$ such that for all $i \in \Z$,
\begin{enumerate}
\item $x_{[q_i +1, q_i + m]} = w^{4p}_{n_j}$, where $m = |w^{4p}_{n_j}|$,
\item $|w_{n_j}^p| \leq q_{i+1} - q_i \leq 2 |w^{4p}_{n_j}|$. \end{enumerate}
Hence, the sequence $(q_i)_{i \in \Z}$ is relatively dense in $\Z$ with frequency at least $1/(8p|w_{n_j}|)$. The separation by $|w_{n_j}^p|$ ensures that each such occurence of $w_{n_j}^{4p}$ contributes $(|w_{n_j}^p| + 1)$ occurences of $3$-blocks $www$, satisfying $|w|= p |w_{n_j}|$ within $x$.
This yields
\[
\mu(\mathbb{X}_{\alpha}(p |w_{n_j}|) \geq \frac{1}{8p|w_{n_j}|} p |w_{n_j}| = \frac{1}{8}.
\]
Taking the $\limsup$ as $j \to \infty$ yields \eqref{eq:PROP:unbounded-continued-fraction:goal}.
\end{proof}

\begin{coro} \label{CORO:unbounded-continued-fraction}
Assume that $\X'$ has period $p$ and that $\alpha = [a_1, a_2, a_3, \ldots]$ has unbounded continued fraction expansion, that is, \begin{equation}
\limsup_{n \to \infty} a_n = \infty.
\end{equation}
 Then, $H_{\omega}$ has no eigenvalues for $\rho$-almost every $\omega \in \mathbb{X} \times \mathbb{X}'$.
\end{coro}

\begin{coro}
Let $x'$ be a periodic sequence and $t(\alpha,\theta) = s(\alpha,\theta) \times x'$. Then, $H_{t(\alpha,\theta)}$ has no eigenvalues for Lebesgue almost every $\alpha$ and almost every $\theta$.
\end{coro}

\begin{proof}
Let $(\mathbb{X}',S')$ be the periodic subshift generated by $x'$. By \cite[Thm.~29]{K64}, Lebesgue almost every $\alpha \in (0,1)$ has an unbounded continued fraction expansion and hence Corollary~\ref{CORO:unbounded-continued-fraction} applies to $\mathbb{X}_{\alpha} \times \mathbb{X}'$ for almost every $\alpha$. Since it holds for every ergodic measure $\rho$ and $\nu = \mu \times \mu'$ is a finite linear combination of those by Choquet's theorem (see, e.g., \cite{Phelps2001ChoqBook}), it also follows that $H_{\omega}$ has no eigenvalues for $\nu$-almost every $\omega \in \mathbb{X}_{\alpha} \times \mathbb{X}'$. Hence,
\begin{align*}
0 &= \nu ( \{ t(\alpha,\theta) \in \mathbb{X}_{\alpha} \times \mathbb{X}' : H_{t(\alpha,\theta)} \mbox{ has eigenvalues} \} )
\\ &= \frac{1}{p} \mu(\{ s(\alpha,\theta) \in \mathbb{X}_{\alpha} :  H_{t(\alpha,\theta)} \mbox{ has eigenvalues}\})
\\ &= \frac{1}{p} \Leb (\{ \theta \in [0,1) : H_{t(\alpha,\theta)} \mbox{ has eigenvalues} \}).
\end{align*}
Thus, for almost every $\alpha$ and almost every $\theta$ it holds that $H_{t(\alpha, \theta)}$ has no eigenvalues.
\end{proof}

\subsubsection{Quasi-Sturmian Subshifts}
Suppose the subshift $(\mathbb{X},S)$ is \hbox{(quasi-)}Sturmian, that is, $\mathbb{X}$ is minimal and there exist $m,n_0 \in \N$ such that $\pf(n) = n+m$ for all $n \geq n_0$, where $\pf(n) = \# \mc L_n(\X)$ denotes the complexity function of $\mathbb{X}$. 
Every quasi-Sturmian subshift $(\mathbb{X},S)$ is a Boshernitzan subshift \cite[Cor.~1]{DL06c}. Compare \cite{Cass97} and \cite{DL03} for the following result.

\begin{prop}
A subshift $(\mathbb{X},S)$ is quasi-Sturmian if and only if there exists a Sturmian subshift $(\mathbb{X}_{\alpha},S)$ and an aperiodic substitution $\phi$ on $\{0,1\}$ such that every $x \in \mathbb{X}$ can be written as
\[
x = S^j \phi(y),
\]
for some $y \in \mathbb{X}_{\alpha}$ and $j \in \Z$.
\end{prop}

In that situation, we call $(\alpha, \phi)$ a \emph{production pair} of $(\mathbb{X},S)$. 
Recall that we denote by $\rho$ an ergodic measure on $(\mathbb{X}\times \mathbb{X}',T)$.

\begin{prop}
Let $(\mathbb{X},S)$ be a quasi-Sturmian subshift with production pair $(\alpha,\phi)$. Assume that $(\mathbb{X}',S')$ has period $p$ and that $\alpha = [a_1,a_2,a_3,\ldots]$ satisfies $\limsup_{n \to \infty} a_n \geq 4p$. Then, for $\rho$-almost every $\omega \in \mathbb{X} \times \mathbb{X}'$, $H_\omega$ has no eigenvalues.
\end{prop}

We omit the proof as it is completely analogous to the proof of Proposition~\ref{PROP:unbounded-continued-fraction}. This is because the structure used in the proof of Proposition~\ref{PROP:unbounded-continued-fraction} is preserved under the substitution $\phi$.

\begin{remark}
Let $(\alpha,\phi)$ be a production pair of the quasi-Sturmian subshift $(\mathbb{X},S)$ and let $p \in \N$ be the period of $(\mathbb{X}',S')$. In the special case that $p$ divides both $|\phi(0)|$ and $|\phi(1)|$, the system $(\mathbb{X}\times \mathbb{X}',T)$ decomposes into $p$ minimal components. This follows from the fact that every aperiodic substitution acting on a subshift over a binary alphabet is recognizable \cite[Thm.~3.1]{BSTY}. Given $x' \in \mathbb{X}'$ and $u \in \mc L_n$, with $n$ large enough, each of the sets $[u] \times \{(S')^j x'\}$, for $1\leq j \leq p$, is contained in precisely one of the $T$-minimal components. Hence, each of the $T$-minimal components has a complexity function that eventually coincides with the complexity function of $\mathbb{X}$, and thereby comprises a quasi-Sturmian subshift for which uniform absence of Schr\"odinger eigenvalues and fractional Hausdorff continuity of the spectrum is known \cite{DL03}.
\end{remark}

\subsection{Periodic and Bernoulli} \label{subsec:per+bern}

The discussion presented in this subsection is motivated by the following problem. In the study of the Anderson model (i.e., the potential is given by i.i.d.\ random variables), a fundamental result states that the spectrum of the random operator is almost surely equal to an explicit set. This set is given by the Minkowski sum of the spectrum of the Laplacian and the topological support of the single-site distribution. Thus, we have the pleasant feature that the almost sure spectrum of the sum of the Laplacian and a random potential is given by the sum of the spectrum of the Laplacian and the almost sure spectrum of the random potential. As we are generally interested in this paper in retaining crucial spectral features after the addition of a periodic background potential, the specific question we are facing here is whether there is an explicit description\footnote{By ``explicit'' we mean effective. For example, for any given $E$, can we decide in finite time whether it belongs to the almost sure spectrum? Also, can we answer questions of the following type: does the almost sure spectrum have only finitely many gaps?} of the almost sure spectrum of a Schr\"odinger operator whose potential is given by a sum of a periodic term and a term of Anderson type. This question appears to be open and surprisingly difficult for periods greater than one. What we accomplish in this section is to provide an answer in the case of period two (and finitely supported single-site distribution) and to explain why the natural extension of the period-two result fails in the case of period three.  In fact, the relevant question is purely topological in nature, that is, the almost-sure spectrum only depends on the support of the single-site measure, not on the particular choice of probabilities.

In the following we consider the product space $(\mathbb{X} \times \Z_p, T)$, where $\mathbb{X} = \mc A^\Z$ is the full shift on $m \in \N$ symbols and $T \colon (x,j) \mapsto (Sx, j + 1)$, with addition modulo $p$ in the second coordinate. In this case, $(\mathbb{X} \times \Z_p,T)$ is (topologically conjugate to) a subshift of finite type. We equip the full shift with the Bernoulli measure $\mu = \mu_0^\Z$ for a measure $\mu_0$ on $\mc A$ which can be chosen to satisfy $\mu_0(\{a\})>0$ for all $a \in \mc A$ without loss of generality. One can specify a family of random potentials by fixing a locally constant sampling function $f$ of window size one. Up to shifting, we may assume  that $f$ is of the form 
\begin{equation} \label{eq:randper:sfwindowlength1}
f(x,j) = g(x_0,j)
\end{equation}
for some $g:\mc A \times \Z_p \to \R$.

In what follows, we further make the simplifying assumption that $p = 2$, such that $\Z_p = \Z_2$ is two-periodic. 
In this case it turns out that it suffices to consider the $m^2$ two-periodic sequences
\[
\omega_{ab} = \left( (ab)^{\Z},0 \right)
\]
for $a,b \in \mc A$. 
The individual spectra can be determined using the trace of the corresponding monodromy matrix. 
More precisely, let
\[
M_E(a,b) = M_E[(b,1)] M_E[(a,0)] = \begin{bmatrix}
(E - g(b,1))(E-g(a,0)) - 1 & g(b,1) - E
\\ E - g(a,0) & -1
\end{bmatrix},
\]
the trace of which is given by $P_{a,b}(E) := (E - g(b,1))(E-g(a,0)) - 2$.
The spectrum
\[
\sigma(H_{\omega_{ab}}) = \{ E \in \R : P_{a,b}(E) \in [-2,2] \}
\]
is given by the union of two intervals that can be calculated explicitly.

Denoting by $\mu'$ the normalized counting measure on $\Z_2$ it is straightforward to verify that $\nu = \mu \times \mu'$ is an $T$-ergodic measure. Indeed, this follows easily from the fact that $\mu$ is $S^2$-ergodic.

\begin{theorem}
\label{THM:2-periodic-spectrum}
For $\nu$-almost every $\omega \in \mathbb{X} \times \Z_2$, the spectrum of the corresponding Schr\"{o}dinger operator is given by
\begin{equation} \label{eq:2-periodic-spectrum}
\sigma(H_{\omega}) = \bigcup_{a,b \in \mc A} \sigma(H_{\omega_{ab}}).
\end{equation}
Moreover, one has
\begin{equation} \label{eq:2-periodic-spectrumUniversalIncl}
\sigma(H_{\omega}) \subseteq  \bigcup_{a,b \in \mc A} \sigma(H_{\omega_{ab}}).
\end{equation}
for all $\omega \in \X \times \Z_2$.
\end{theorem}

\begin{remark}
As one can readily see from the proof below, the full-measure set on which \eqref{eq:2-periodic-spectrum} holds contains every $\omega \in \X \times \Z_2$ that has a dense $T$-orbit. In particular, the almost-sure spectrum of $H_\omega$ depends only on the support of the single-site distribution $\mu_0$.
\end{remark}

\begin{proof}[Proof of Theorem~\ref{THM:2-periodic-spectrum}]
The inclusion $\sigma(H_{\omega_{ab}}) \subseteq \sigma(H_\omega)$ for all $a,b \in \mc A$ and almost every $\omega$ follows from strong approximation and the fact that almost every $\omega$ has a dense orbit. It therefore suffices to prove \eqref{eq:2-periodic-spectrumUniversalIncl} for all $\omega$. We will show that
\[
\bigcap_{a,b \in \mc A} \varrho(H_{\omega_{ab}}) \subseteq \varrho(H_{\omega})
\]
for all $\omega \in \mathbb{X} \times \Z_2$. To that end, assume that $E \in \varrho(H_{\omega_{ab}})$ for all $a,b \in \mathcal{A}$. Then $M_E(a,b)$ is hyperbolic for all $a,b \in \mathcal{A}$. For a moment, let us fix $a,b \in \mc A$ and set $x = x_b = E - g(b,1)$ and $y = y_a = E - g(a,0)$. The expanding eigendirection of $M_E(a,b)$ is given by $v^+ = (v^+_1, 1)^\top$ and the contracting eigendirection by $v^- = (v^-_1, 1)^\top$ where
\[
v_1^{\pm} = v^{\pm}_1(a,b) = \frac{x}{2} \left( 1 \pm \sqrt{\frac{xy - 4}{xy}} \right) .
\]
The statement $E \in \varrho(H_{\omega_{ab}})$ is equivalent to $xy \in \R \setminus [0,4]$. For a moment let us fix $y \neq 0$.
A direct calculation yields that $v_1^-$ is monotonically decreasing in $x$ and that $\lim_{x \to \pm \infty} v_1^- = 1/y$. From the boundary cases $v_1^- = x/2 = 2/y$ for $xy = 4$ and $v_1^- = 0$ for $x = 0$, we infer that $v_1^-$ lies strictly between $0$ and $2/y$ for all $x$ with $xy \in \R \setminus [0,4]$.
Let
\begin{align*}
y_+ & = \inf \{ y_a : a \in \mc A, y_a > 0\}
\\y_- & = \sup \{ y_a : a \in \mc A, y_a < 0 \}
\end{align*}
be the smallest positive and the largest negative value of $y$, respectively where we adpot the conventions $\inf\emptyset = \infty$ $\sup\emptyset = -\infty$ to deal with cases in which one of the sets of $y$'s is empty. If $y_a = y_+$, then $E \in \varrho(H_{\omega_{ab}})$ implies that $x_b < 0$ or $x_b > 4/y_+$ for all $b \in \mc A$. Similarly we obtain that $x_b > 0$ or $x_b < 4/y_-$. Note that this is still valid for the degenerate cases $y_+ = \infty$ and $y_- = - \infty$. In summary, we have
\[
x_b < \frac{4}{y_-} \quad \mbox{ or } \quad
x_b > \frac{4}{y_+}
\]
for all $b \in \mc A$. For a moment, assume that $x_b,y_a > 0$. Then,
\[
v_1^+ > \frac{x_b}{2} > \frac{2}{y_+}
\quad
\mbox{and} \quad
0 < v_1^- < \frac{2}{y_a} \leq \frac{2}{y_+} .
\]
Exhausting all possible cases in a similar manner we obtain that
\begin{equation}\label{EQ:v-dichotomy}
v_1^-(a,b) \in \left( \frac{2}{y_-}, \frac{2}{y_+} \right) \quad
\mbox{and}
\quad v_1^+(a,b) \in \R \setminus \left[  \frac{2}{y_-}, \frac{2}{y_+} \right]
\end{equation}
for all $a,b \in \mc A$.
Identify in the following all vectors with their representatives on the real projective space $\mathbb{P}^1$. The dichotomy in \eqref{EQ:v-dichotomy} amounts to the observation that there exists an open interval $I \subseteq \mathbb{P}^1$ such that $M_E(a,b)$ acts as a contraction on $I$ for all $a,b \in \mc A$. By \cite[Thm.~2.2]{ABY} this implies that every cocycle that is built from the matrices $\{ M_E(a,b) \}$ is uniformly hyperbolic. Hence, $E \in \varrho(H_{\omega})$ for all $\omega \in \mathbb{X} \times \mathbb{X}'$ by Johnson's theorem \cite{DF1, Johnson1986JDE, ZhangZ2020JST}.
\end{proof}

\begin{proof}[Proof of Theorem~\ref{t:2per+rand}]
This follows immediately from Theorem~\ref{THM:2-periodic-spectrum} and a suitable choice of sampling function.
\end{proof}

\begin{remark}
At first, we might expect that we could get a similar result for larger periods as well. However, if $\Z_p = \Z_3$ is $3$-periodic, the situation is already different. For a concrete example, consider $\mathbb{X} = \{0,3\}^\Z$, 
let
\[
\omega_{abc} = ((abc)^{\Z},0),
\]
for $a,b,c \in \{0,3\}$, 
and let the corresponding sampling function be given by $g(a,j) = g_1(a) + g_2(j)$, with $g_1(a) = a$ for $a \in \{0,3\}$, and 
\[
g_2 \colon \begin{cases}
0 \mapsto 0,
\\1 \mapsto 2,
\\2 \mapsto 3.
\end{cases}
\]
By explicit calculations we obtain that the spectrum of $H_{\omega}$ with the sequence $\omega = ((000333)^\Z,0 )$ contains the interval $[1.385,1.423]$ which is disjoint from $\sigma(H_{\omega_{abc}})$, for all $a,b,c \in \{0,3\}$. Hence,
\[
\bigcup_{a,b,c \in \mc A} \sigma(H_{\omega_{abc}}) \subsetneq \sigma(H_{\omega^\ast}),
\]
where $\omega^\ast$ is any point with a dense orbit in $\mathbb{X} \times \Z_3$. The natural analogue of Theorem~\ref{THM:2-periodic-spectrum} therefore fails in the $3$-periodic case.
\end{remark}

Let us return to the case of general periods. We consider again $f$ of the form \eqref{eq:randper:sfwindowlength1} where the periodic factor is given by $\Z_p$, with $p \in \N$.
As before, let $\nu = \mu \times \mu'$, with $\mu'$ the normalized counting measure on $\Z_p$.

The Lyapunov exponent is given by
\[L(E) = L(E;\nu) = \lim_{n\to\infty} \frac1n \int_\Omega \log\|A_E^n(\omega)\| \, d\nu(\omega).\]
One is naturally interested in proving positivity of $L$ on a rich set of energies as a starting point towards a proof of Anderson localization. Indeed, we will show this happens as soon as the function $g$ assumes different values.

\begin{theorem} \label{t:randper:posle}
If $g$ is nonconstant, then $L(E)>0$ for all but finitely many $E \in \R$.
\end{theorem}

The following lemma is helpful. This does not need a random or ergodic setting. For $a \in \R$ and $w = a_1\ldots a_n$, define $M_E$ and $M_E$ by \eqref{eq:boshper:MEadef} and \eqref{eq:boshper:MEwdef} with $g(a) =a$.

\begin{lemma} \label{lem:poorMansBorgMarchenko}
Let $u$ and $v$ denote words. If $M_z(u) = M_z(v)$ for all $z \in \C$, then $u=v$. That is $|u|=|v|=:\ell$ and $u_j = v_j$ for all $1\le j \le \ell$.
\end{lemma}

\begin{proof}
If $M_z(u) \equiv M_z(v)$, it is clear that $|u| = |v|$ by degree considerations. We now proceed by induction on $\ell = |u|=|v|$. The claim is trivial when $\ell=1$ and follows from the observation
\[ \begin{bmatrix}
* & u_2 -z \\
z-u_1 & * \end{bmatrix}
=
M_z(u_1u_2) =
M_z(v_1v_2) = \begin{bmatrix}
* & v_2 -z \\
z-v_1 & * \end{bmatrix}
 \]
when $\ell = 2$. Now assume $\ell \geq 3$ and $M_z(u) \equiv M_z(v)$. Write $u'=u_1\ldots u_{\ell-1}$. Notice that for any $w$, the degree of $[M_z(w)]_{21}$ is $|w|-1$. Thus, we have the following (making use of $[M_z(u')]_{11} = [M_z(u)]_{21}$.
\begin{align*}
[M_z(u)]_{11} 
 = (z-u_\ell)[M_z(u')]_{11} -[M_z(u')]_{21} 
 = (z-u_\ell)[M_z(u)]_{21}  + O(z^{\ell-2}).
\end{align*}
Perform the analogous calculation for $[M_z(v)]_{11}$, use $[M_z(v)]_{21} \equiv [M_z(u)]_{21}$, and compare the $z^{\ell-1}$ terms to see that $u_{\ell} =v_{\ell}$. The result follows by induction.\end{proof}

\begin{proof}[Proof of Theorem~\ref{t:randper:posle}]
Since $L$ is positive away from the almost-sure spectrum, it suffices to show that $L$ can only vanish on a discrete set. Consider the regrouped alphabet $\hat{\mc A} = {\mc A}^p$ and $\hat\mu_0 = \mu_0^p$. For $a = (a_1,\ldots,a_p) \in \hat{\mc A}$ and $E \in \C$, define 
\begin{align*}
 \hat{A}_E(a)
 & = M_E(a_p,p) \cdots M_E(a_1,1) \\
& = \begin{bmatrix} E - g(a_p,p) & -1 \\ 1 & 0 \end{bmatrix}  \begin{bmatrix} E - g(a_{p-1}, p-1) & -1 \\ 1 & 0 \end{bmatrix}  \cdots \begin{bmatrix} E - g(a_1,1) & -1 \\ 1 & 0 \end{bmatrix} ,
\end{align*}
and consider the induced cocycle $(T, \hat{A}_E)$ on $\hat{A}^\Z$ with ergodic measure $\hat\mu = \hat\mu_0^\Z$ and its associated Lyapunov exponent
\begin{equation}
\hat{L}(E) = \lim_{n\to\infty} \frac1n \int_{\hat{\mc A}^\Z} \log\| \hat{A}_E^n(\hat\omega) \| \, d\hat\mu(\hat\omega).
\end{equation} 
For $a,b \in \hat{\mc A}$, Lemma~\ref{lem:poorMansBorgMarchenko} implies that the commutator $[M_E(a),M_E(b)]$ vanishes identically in $E \in \C$ if and only if $(g(a_1,1),\cdots g(a_p,p)) = (g(b_1,1),\ldots,g(b_p,p))$. Thus, by the assumption on $g$, there exist $a,b \in \hat{\mc A}$ and $E \in \C$ such that $M_E(a)$ and $M_E(b)$ do not commute, and hence $\hat{L}$ is positive away from a discrete set by the abstract Furstenberg criterion from \cite{BDFGVWZ2019JFA}. The result then follows by using interpolation to note that $\hat{L}(E) = pL(E)$.
\end{proof}

As one can see from the regrouping construction in the proof, the models discussed here are special cases of random word models, which are known to exhibit Anderson localization. Indeed, the result described in Theorem~\ref{t:randper:posle} was already known. We give the proof here, since it is much simpler than the argument from \cite{DamSimsStolz2004JFA}. However, the spectrum (as a set) is not explicitly identified as in Theorem~\ref{THM:2-periodic-spectrum} in complete generality.

\section{Periodic and One-Frequency Quasi-Periodic} \label{sec:qpper}

We now turn to the second main family of examples of product systems: products of circle rotations and translations on finite cyclic groups. To keep the length of the paper in check, we do not attempt an exhaustive survey of all possible results in this scenario. Rather, we look at a selection of results that we consider interesting. 

The motivating example is that of a quasi-periodic potential with a periodic background:
\[V(n) = V_x(n) + V_\per(n),\]
where $V_\per$ has period $p$ and $V_x(n) = f_1(n\alpha+x)$ for some $f_1 \in C(\T,\R)$, $x \in \T := \R/\Z$, and $\alpha \in \T$ irrational. One can clearly encode this via the product system $(\Omega,T)$ where
\begin{equation} \label{eq:qpper:TZpsystDef}\Omega = \T \times \Z_p, \quad T(x,k)= (x+\alpha,k+1), \quad x \in \T, \ k \in \Z_p.
\end{equation} One generates $V_x+V_\per$ as $V_{(x,0)}(n) = f(T^n(x,0))$ via the sampling function $f(x,k) = f_1(x)+f_2(k)$ where $f_2(k) = V_\per(\tilde k)$ for any representative $\tilde k$ of the residue class $k \in \Z_p$. We will be mainly interested in this case, but we can be a bit more general, considering for example trigonometric polynomials on the space $\Omega$ (see Def.\ \ref{def:trigpoly} below).

\subsection{Generalities}
Given $\alpha \in \T$ irrational and $p \in \Z$, we consider the associated product system as in \eqref{eq:qpper:TZpsystDef}. 
As a consequence of the general discussion in Appendix~\ref{app:StrictErgProd}, let us note the following basic facts about this particular product system.

\begin{prop} \label{PROP:torusXZpbasics}
Given $p \in \N$ and $\alpha \in \R \setminus \Q$, let $\Omega = \T \times \Z_p$ and $T(x,k)=(x+\alpha,k+1)$.
\begin{enumerate}
\item[{\rm(a)}] $\Omega$ is a compact abelian group. 
\item[{\rm(b)}] $(\Omega,T)$ is minimal. 
\item[{\rm(c)}] $(\Omega,T,\mu)$ is ergodic, where $\mu=\mu_1\times \mu_2$, $\mu$ denotes Lebesgue measure on $\T$, and $\mu_2$ denotes normalized counting measure on $\Z_p$.
\item[{\rm(d)}]$(\Omega,T)$ is uniquely ergodic with unique invariant measure $\mu$ as in part {\rm(c)}.
\end{enumerate}
\end{prop}

Let $f \in C(\Omega,\R)$ be given. With the help of Proposition~\ref{PROP:torusXZpbasics}, we make a few observations. First, by minimality of $(\Omega,T)$ and continuity of $f$, there is a uniform set $\Sigma = \Sigma_{f,\alpha}$ with $\Sigma=\sigma(H_{f,\alpha,\omega})$ for all $\omega \in \Omega$. 

Recall the one-step cocycle map
\begin{equation}
A_z(\omega) = \begin{bmatrix}
z-f(T\omega) & -1 \\ 1 & 0
\end{bmatrix}
\end{equation}
and the associated  Lyapunov exponent
\[L(z) = L(z,f,\alpha) = \lim_{n\to\infty} \frac{1}{n} \int_\Omega \log\|A_z^n(\omega)\| \, d\mu(\omega). \]

One of the main ideas in the analysis of this family of product systems is to pass from $\Omega = \T \times \Z_p$ to $\T$ by regrouping. Concretely, define $B_z = B_{z,f,\alpha}:\T \to \SL(2,\C)$ by
\begin{align}
\nonumber
B_z(x) 
& = \prod_{j=p}^1 \begin{bmatrix}
z- f(T^j(x,0)) & -1 \\ 1 & 0
\end{bmatrix} \\
\label{eq:qpper:BzRegroupedDef}
&= A_z^p((x,0)), \quad x \in \T.
\end{align}
This map has iterates 
\begin{align}
\nonumber
B_z^n(x)& = B_z(x+(n-1)p\alpha) \cdots B_z(x+p\alpha)B_z(x) \\
\label{eq:qpper:BzRegroupedDef2}
& = A_z^{np}((x,0)).
\end{align}
Denote the corresponding Lyapunov exponent by $\widetilde{L}(z) = \widetilde{L}(z,f,\alpha)$:
\[
\widetilde{L}(z) = \lim_{n\to\infty} \frac{1}{n} \int_0^1 \log\|B_z^n(x)\| \, dx.
\]

\begin{prop} \label{prop:qpper:LvstildeL}
$\widetilde{L}=pL$. In particular, $L(z)>0 \iff \widetilde{L}(z)>0$.
\end{prop}

\begin{proof}
This follows from \eqref{eq:qpper:BzRegroupedDef2}.
\end{proof}

From Proposition~\ref{prop:qpper:LvstildeL}, we define
\[\mathcal{Z} = \set{E : L(E)=0} = \set{E : \widetilde{L}(E)=0}.\]

\subsection{Consequences of Global Theory}

Let us briefly recall the terminology from Avila's global theory of one-frequency analytic cocycles \cite{Avila2015Acta}. Let $\alpha \in \R\setminus \Q$ be given, and suppose $B:\T \to \SL(2,\R)$ is real-analytic with analytic extension to a strip $\T_s = \{  z : |\Im(z)| <s\}$ for some $s>0$. For each $\varepsilon\in \R$ with $|\varepsilon|<s$, one may consider the cocycle $B_\varepsilon := B(\cdot + \mi\varepsilon)$ and the associated Lyapunov exponent
\begin{align}
L(B_\varepsilon,\alpha)
& =\lim_{n\to\infty} \frac{1}{n} \int_\T \log\|B_\varepsilon(x+(n-1)\alpha) \cdots B_\varepsilon(x+\alpha)B_\varepsilon(x)\|\, dx \\[1mm]
& =\lim_{n\to\infty} \frac{1}{n} \int_\T \log\|B_\varepsilon^n(x)\|\, dx.
\end{align}

\begin{theorem}[Avila (2015) \cite{Avila2015Acta}] \label{t:avilagt}
Given $\alpha \in \R \setminus \Q$ and $B:\T \to \SL(2,\R)$ with analytic extension to $\T_s$, the function $\Lambda:\varepsilon \mapsto L(B_\varepsilon,\alpha)$ enjoys the following properties.
\begin{enumerate}
\item[{\rm(a)}] $\Lambda$ is continuous, convex, and piecewise affine on $(-s,s)$.
\item[{\rm(b)}] {\rm(}quantization of acceleration{\rm)}. For all $|\varepsilon|<s$, the acceleration
\[ \omega(B,\varepsilon) := \lim_{t\downarrow 0} \frac{1}{2\pi t}(\Lambda(\varepsilon+t) - \Lambda(\varepsilon))\]
exists and must lie in $\Z$.
\end{enumerate}
\end{theorem}

In view of Theorem~\ref{t:avilagt}, one classifies cocycle maps as follows.

\begin{definition}
With $B$ and $\alpha$ as above, we say that the cocycle $(B,\alpha)$ is:
\begin{itemize}
\item \emph{subcritical} if for some $\delta>0$, $L(B_\varepsilon,\alpha)=0$ for all $|\varepsilon|<\delta$;
\item \emph{critical} if $L(B,\alpha)=0$, but $(B,\alpha)$ is not subcritical; and

\item \emph{supercritical} if $L(B)>0$ but $(B,\alpha)$ is not uniformly hyperbolic. 
\end{itemize}
\end{definition}

As discussed in the introduction, we will consider periodic decorations of quasi-periodic potentials generated by trigonometric polynomials, and this is most commonly accomplished with the addition of a periodic background. The arguments can handle a more general situation, which we now formulate precisely.

\begin{definition} \label{def:trigpoly}
Recall that a \emph{character} of a topological group $G$ is a continuous homomorphism $G \to \mathbb{S}^1 = \{z \in \C  : |z|=1\}$ and a \emph{trigonometric polynomial} is a linear combination of characters. We write $\trigpoly(G)$ for the set of trigonometric polynomials on $G$.
\end{definition}

The following well known characterization of trigonometric polynomials on $\T \times \Z_p$ will be helpful.

\begin{prop} \label{prop:qpper:trigpoly}  One has $f \in \trigpoly(\T \times \Z_p)$  if and only if $f(\cdot,k) \in \trigpoly(\T)$ for each $k \in \Z_p$.
\end{prop}

\begin{proof}
This is well known and not hard to show using unitarity of the discrete Fourier transform. For the reader's convenience, we give the arguments. The characters of $\T\times \Z_p$ are of the form
\begin{equation}
 \chi_{m,\ell} :(x,k) \mapsto \me^{2\pi \mi (mx+k\ell/p)}, \quad m \in \Z, \ \ell \in \Z_p. \end{equation}
Thus, if $f \in \trigpoly(\T \times \Z_p)$, then
\begin{equation} \label{eq:qpper:trigPolyAsSumOfchars} f = \sum_{m \in \Z} \sum_{\ell \in \Z_p} c_{m,\ell} \chi_{m,\ell}  \end{equation}
for suitable coefficients $\{c_{m,\ell}\}$,
which certainly implies $f(\cdot, k) \in \trigpoly(\T)$ for each $k$.

Conversely, if $f(\cdot,k) \in \trigpoly(\T)$ for each $k$, write
\[f(x,k) = \sum_{m \in \Z} \hat{c}_{m,k} \me^{2\pi \mi m x} \]
for some coefficients $\{\hat{c}_{m,k}\}$. To write $f$ in the form \eqref{eq:qpper:trigPolyAsSumOfchars}, define
\[  c_{m,\ell} = \frac{1}{p} \sum_{k' \in \Z_p} \me^{-2\pi \mi k'\ell/p} \hat{c}_{m,k'}\]
for each $m \in \Z$, $\ell \in \Z_p$, and note that
\begin{align*}
\sum_{m \in \Z} \sum_{\ell \in \Z_p} c_{m,\ell} \chi_{m,\ell}(x,k)
& = \sum_{m \in \Z} \sum_{k' \in \Z_p} \sum_{\ell \in \Z_p} \frac{1}{p} \hat{c}_{m,k'} \me^{2\pi \mi mx} \me^{2\pi \mi \ell(k-k') /p} \\
& = \sum_{m \in \Z} \hat{c}_{m,k} \me^{2\pi \mi mx} \\
& = f(x,k),
\end{align*}
as desired. 
\end{proof}

In view of Proposition~\ref{prop:qpper:trigpoly}, we can identify trigonometric polynomials on $\T \times \Z_p$ with $p$-tuples of trigonometric polynomials on $\T$. If $f:\Omega \to \R$ is a trigonometric polynomial, we call $f^{[k]} = f(\cdot,k)$ a \emph{component} of $f$.

\begin{theorem} \label{t:qpper:supercritical}
Let $\alpha \in \R \setminus \Q$ and $p \in \N$ be given, and suppose $f$ is a real-valued trigonometric polynomial on $\Omega  = \T \times \Z_p$ such that no component of $f$ vanishes identically.
 If $|\lambda|$ is sufficiently large, then the cocycle $B_{E,\lambda f,\alpha}$ defined by \eqref{eq:qpper:BzRegroupedDef} is supercritical for all $E \in \Sigma$. Indeed, one has
\begin{equation} \label{eq:qp+per:lelb}
\widetilde{L}(E,\lambda f,\alpha) \geq \frac{1}{2}p \log |\lambda|
\end{equation}
for all $E \in \Sigma$ and all $|\lambda|$ sufficiently large.
\end{theorem}

\begin{proof}
Let $f$ be a real trigonometric polynomial on $\Omega$. By Proposition~\ref{prop:qpper:trigpoly}, $f^{[k]} := f(\cdot,k)$ is a trigonometric polynomial on $\T$, and hence we may define $d_k =\deg(f^{[k]})$ for $k \in \Z_p$  and
\[ d=\sum_{k\in\Z_p} d_k. \]

To simplify notation, we view  $f$, $\lambda$, and $\alpha$ as fixed and suppress them from the notation.
For $E,\varepsilon \in \R$, let $B_{E,\varepsilon}= B_E(\cdot + \mi\varepsilon)$ denote the $p$-step cocycle map $B_{E}$ with complexified phase.

Since the potential is real-valued, $L$ is even in $\varepsilon$. By considering $|\varepsilon|$ large, we see
\[
\me^{-2\pi d|\varepsilon|}B_{E,\varepsilon} =
\prod_{j=p-1}^0
\begin{bmatrix}
c_j\lambda + o(1) & o(1) \\ o(1) & 0
\end{bmatrix},
 \]
for suitable constants $c_k \neq 0$, and thus, denoting $\hat{c} = \log|c_0\cdots c_{p-1}|$, we have
\begin{equation} \label{eq:qpper:LElargeEps}
L(B_{E,\varepsilon},p\alpha) = p\log|\lambda| + \hat{c}+ 2\pi d|\varepsilon| \text{ for all } |\varepsilon| \text{ sufficiently large.}
\end{equation}
By quantization, evenness in $\varepsilon$, and convexity, \eqref{eq:qp+per:lelb} follows for large enough $|\lambda|$.  By Johnson's theorem \cite{DF1, Johnson1986JDE, ZhangZ2020JST}, the cocycle $(B_E,p\alpha)$ cannot be uniformly hyperbolic for $E \in \Sigma$, which concludes the argument.
\end{proof}
 
\begin{remark} \label{rem:qpper:super} Let us make a few comments about Theorem~\ref{t:qpper:supercritical}.
\begin{enumerate}
\item [(a)] The assumption that no component of $f$ vanishes is essential. Indeed, consider the case $p=2$ and $f \in \trigpoly(\T\times \Z_2)$ for which $f^{[1]} \equiv 0$ and $f^{[0]}$ is some nonconstant real-valued trigonometric polynomial on $\T$. For energy $E = 0$, one sees immediately
\[ B_0(x) = \begin{bmatrix} 0 & -1 \\ 1 & 0  \end{bmatrix}  \begin{bmatrix} -f^{[0]}(x) & -1 \\ 1 & 0  \end{bmatrix}
=\begin{bmatrix} -1 & 0 \\ -f^{[0]}(x) & -1  \end{bmatrix},
\]
leading to
\begin{equation}
B_0^n(x) = (-1)^n \begin{bmatrix} 1 & 0 \\ \sum_{j=0}^{n-1} f^{[0]}(x+2j\alpha) & 1 \end{bmatrix}
\end{equation}
This suffices to show that $0$ is a generalized eigenvalue of $H_{f,\alpha,\omega}$ for every $x \in \T$ and that $\widetilde{L}(0) = L(0) = 0$, so the associated $2$-step cocycle is not supercritical, regardless of the size of $f$.

\item[(b)] On the other hand, in the case in which one considers the product system associated with the sum of a quasi-periodic potential generated by $f_0 \in \trigpoly(\T)$ and a periodic background $V_\per$, the associated sampling function is of the form $f(x,k) = f_0(x)+f_1(k)$. In particular, the obstruction noted in (a) cannot occur in this setting.
\end{enumerate}
\end{remark}

One can also prove subcriticality at small coupling under suitable assumptions on $f$.

\begin{theorem} \label{t:qpper:subcritical}
Let $\alpha \in \R\setminus \Q$ and $p \in \N$ be given, and suppose $f$ is a real-valued trigonometric polynomial on $\Omega = \T \times \Z_p$ such that every component of $f$ has the same degree, $d_0$.
 If $|\lambda|$ is sufficiently small, then the cocycle $B_{E,\lambda f,\alpha}$ defined by \eqref{eq:qpper:BzRegroupedDef} is subcritical for all $E \in \Sigma$. In particular, there is a constant $c = c(f)>0$ such that  $B_{E,\lambda f,\alpha}$ is subcritical on the spectrum whenever $|\lambda| \leq c$.
\end{theorem}


\begin{remark}
(a) The constant $c(f)$ can be made explicit for specific examples. The dependence on $f$ mainly enters through the Herman radius (see \eqref{eq:qp+per:hermanRadDef}), which can be computed or estimated for specific choices of trigonometric polynomials.

(b) Notice that this result applies in the case in which one considers periodic perturbations of a quasi-periodic operator whose potential is generated by a trigonometric polynomial on $\T$, since every component of $f$ is then a shift of a single fixed trigonometric polynomial, and hence has the same degree; compare the discussion in Remark~\ref{rem:qpper:super}.
\end{remark}

\begin{proof}[Proof of Theorem~\ref{t:qpper:subcritical}]
This is an adapdation of the main argument of \cite{MSW2018JST}. Normalize $f$ by $\|f\|_\infty = 1$ and assume $|\lambda| $ is small; in particular, we assume \textit{a priori} that  $|\lambda|\leq 1$.

Define the Herman radius, $\varepsilon_H = \varepsilon_H(E,\lambda f)$, by
\begin{equation} \label{eq:qp+per:hermanRadDef}
\varepsilon_H(E,\lambda f) = \sup\set{ \varepsilon\geq 0 : \min_{(x,k) \in \Omega}  |\lambda f(x+\mi\varepsilon,k)  - E| \leq 2 }.
\end{equation}

By the assumption $\|f\|_\infty = 1$, and $|\lambda | \leq 1$, one has $\sigma(H_{\lambda f, \alpha, \omega}) \subseteq [-3,3]$ for $\omega \in \Omega$. For $E \in [-3,3]$, one observes the following bound (for large $|\varepsilon|$)
\begin{align*}
|\lambda f(x + \mi\varepsilon,k) - E| 
& \geq |\lambda||f(x+\mi\varepsilon,k)| - |E| \\
& \geq |\lambda||f(x+\mi\varepsilon,k)| - 3 \\
& \geq \delta |\lambda| \me^{2\pi d_0 |\varepsilon|}-3,
\end{align*}
where $d_0$ denotes the degree of $f^{[k]}$ and $\delta > 0$ is a constant that depends on $f$. Consequently, one arrives at
\begin{equation}
|\lambda f(x + \mi\varepsilon,k) - E|  >2 \text{ when }  |\varepsilon| \geq \frac{1}{2\pi d_0} \log\left( \frac{5}{\delta|\lambda|} \right).
\end{equation}
This shows that
\begin{equation}
\varepsilon_H \leq \frac{1}{2\pi d_0} \log\left( \frac{5}{\delta|\lambda|} \right),
\end{equation}
On the other hand, we have already seen in \eqref{eq:qpper:LElargeEps}
that:\footnote{Recall that $d$ denotes the sum of degrees of components of $f$ in that proof.}
\[
L(B_{E,\varepsilon},p\alpha) = p\log|\lambda| + \hat{c} + 2\pi d|\varepsilon| \text{ for all } |\varepsilon| \geq \varepsilon_0,
\]
which holds with $\varepsilon_0 = \varepsilon_H$ by the same argument\footnote{Namely, the definition of the Herman radius ensures that $(B_{E,\varepsilon},p\alpha)$ is uniformly hyperbolic for $\varepsilon >\varepsilon_H$. Since $\Lambda : \varepsilon\mapsto L(B_{E,\varepsilon},p\alpha)$ is affine in a neighborhood of $\varepsilon$ whenever $(B_{E,\varepsilon},p\alpha)$ is uniformly hyperbolic, the slope of $\Lambda$ cannot change above $\varepsilon=\varepsilon_H$.} as in the proof of \cite[Theorem~1.1]{MSW2018JST}. Now, assume some $E$ in the spectrum is not subcritical. Then the acceleration at $\varepsilon=0$ must be at least one, leading to the following estimate via convexity and \eqref{eq:qpper:LElargeEps}
\begin{align*}
L(B_{E,\varepsilon},p\alpha) 
\leq p\log|\lambda| +\hat{c}+2\pi d \varepsilon_H + 2 \pi (\varepsilon-\varepsilon_H), \quad 0 \le \varepsilon\le \varepsilon_H,
\end{align*}
which yields the following by taking $\varepsilon=0$:
\begin{align*}
0 & \leq  p\log|\lambda|+\hat{c} +2\pi (d-1) \varepsilon_H \\
& \leq  p \log|\lambda| +\hat{c} + \frac{d-1}{d_0}[\log(5/\delta) - \log |\lambda|] \\
& = \frac{pd_0 - d + 1}{d_0} \log |\lambda| + \hat{c} +\frac{d-1}{d_0}\log(5/\delta) \\
& = \frac{1}{d_0} \log |\lambda| + \hat{c} +\frac{d-1}{d_0}\log(5/\delta),
\end{align*}
where we used that every component has the same degree to obtain $pd_0 = d$ in the last line.
This in turn implies a lower bound on $|\lambda|$ in terms of $\hat{c}$, $p$, $d$, and $\delta$. Turning this around, if $|\lambda|$ is sufficiently small, every energy in the spectrum must be subcritical, as promised.
\end{proof}

\begin{proof}[Proof of Theorem~\ref{t:qppermain}]
Suppose $f_0$ is a nonconstant trigonometric polynomial on $\T$, and define $f_1(k) = V_\per(\tilde k)$ for any representative $\tilde k$ of $k \in \Z_p$. The desired result then follows by applying Theorems~\ref{t:qpper:supercritical} and \ref{t:qpper:subcritical} with the sampling function $f(x,k) = f_0(x)+f_1(k)$ which is clearly a trigonometric polynomial for which every component is a nonconstant polynomial of the same degree.\end{proof}

\begin{proof}[Proof of Theorem~\ref{thm:AMOper}]
This follows immediately from Theorem~\ref{t:qppermain}.
\end{proof}

\makeatletter
\renewcommand{\thetheorem}{\thesection.\arabic{theorem}}
\@addtoreset{theorem}{section}
\makeatother

\begin{appendix}

\section{Ergodic Measures on Accelerated Systems}
\setcounter{subsection}{1}

Given a uniquely ergodic topological dynamical system $(\X,S)$, the setting of the present paper naturally motivates one to understand the structure of the space of $S^m$-invariant (and $S^m$-ergodic) measures on $\X$. We collect some basic results to that effect here. The results in this section are measure-theoretic analogs of statements from Section~\ref{sec:m&uesys} in the topological setting. Throughout this part of the appendix, assume that $(\X,S)$ is uniquely ergodic with unique invariant measure $\mu$.

\begin{lemma}
\label{LEM:S^n-ergodic-structure}
For every $m \in \N$, there is a measurable subset $A \subseteq \mathbb{X}$ and a number $q \in \N$ dividing $m$ with the following properties. $A = S^q(A)$, and the set of ergodic probability measures on $(\mathbb{X},S^m)$ is given by
\[
\{ q \,\mu |_A \circ S^{-j} : 0 \leq j \leq q-1 \}.
\]
Further, $\mu|_A \circ S^{-q} = \mu|_A$, $\mu(A) = 1/q$ and the set $\cup_{j = 0}^{q-1} S^j (A)$ has full $\mu$-measure.
\end{lemma}

\begin{proof}
Let $\varrho$ be an $S^m$-ergodic probability measure on $\mathbb{X}$. First note that $\varrho \circ S^{-j}$ is $S^m$ ergodic for all $j \in \N$. Indeed, if $B$ is $S^m$-invariant, then so is $S^{-j}B$ and hence $\varrho \circ S^{-j}(B) = \varrho(S^{-j}B) \in \{0,1\}$. Hence, either $\varrho \circ S^{-j} = \varrho$ or $\varrho \circ S^{-j} \perp \varrho$. Let $q \in \N$ be minimal with the property that $\varrho = \varrho \circ S^{-q}$. Since $\varrho$ is $S^m$-invariant, $q$ needs to divide $m$. The measure
\[
\bar{\varrho} = \frac{1}{q} \sum_{j=0}^{q-1} \varrho \circ S^{-j}
\]
is an $S$-invariant probability measure on $\mathbb{X}$ by construction. Due to the unique ergodicity of $(\mathbb{X},S,\mu)$ it follows that $\bar{\varrho} = \mu$. Let $A'$ be a measurable set such that $\varrho(A') = 1$ and $\varrho \circ S^{-j} (A') = 0$ for all $1 \leq j \leq q-1$. The same holds for the set $A = \cap_{j \in \Z} S^{jq}(A')$, which in addition satisfies $S^q(A) = A$. For an arbitrary subset $C \subseteq \mathbb{X}$ we obtain
\[
\mu |_A (C) = \mu(A\cap C) = \bar{\varrho}(A \cap C)
= \frac{1}{q} \varrho(C),
\]
implying that $\varrho = q \, \mu |_A$. In particular, $\mu(A) = \mu|_A(A) = \varrho(A)/q = 1/q$. Further, we find
\[
\mu(\cup_{j=0}^{q-1} S^j(A)) = \frac{1}{q} \sum_{k=0}^{q-1} \varrho (S^{-k}(\cup_{j=0}^{q - 1} S^j(A))) = \frac{1}{q} \sum_{k=0}^{q-1} \varrho(A) = 1.
\]
Suppose there is another $S^m$-ergodic probability measure $\nu$ on $\mathbb{X}$, which is then singular to each of the $\varrho \circ S^{-j}$. The same holds for each of the measures $\nu \circ S^{-j}$ with $j \in \N$. By the same argument as above, there exists $q' \in \N$ such that $\mu = \bar{\nu} := 1/q' \sum_{j=0}^{q'-1} \nu \circ S^{-j}$. On the other hand, $\bar{\nu} \perp \bar{\varrho}$ leading to a contradiction.
\end{proof}

\begin{lemma}
\label{LEM:n-eigenvalues}
For every $m \in \N$, $\me^{2 \pi \mi/m}$ is an eigenvalue of $(\X,S, \mu)$ if and only if $(\mathbb{X},S^m)$ has precisely $m$ ergodic probability measures.
\end{lemma}

\begin{proof}
Suppose $\me^{2 \pi \mi /m}$ is an eigenvalue of $(\mathbb{X},S, \mu)$ with (almost-surely) normalized eigenfunction $f$. For $0 \leq j \leq m-1$, let
\[
A_j = f^{-1}(\{ \me^{2 \pi \mi \alpha} : j/m \leq \alpha < (j+1)/m \}).
\]
By construction, all the $A_j$ are disjoint and we have for $0 \leq j \leq m-2$ that $S(A_j) = A_{j+1}$ as well as $S(A_{m-1}) = A_0$ and hence $\mu(A_j)$ does not depend on $j$. Since $|f| =1$ almost surely, the union $A_0 \cup \ldots \cup A_{m-1}$ has full measure, implying that $\mu(A_j) = 1/m$ for all $0\leq j \leq m-1$. The probability measures $\mu_j = m \, \mu |_{A_j}$ are $S^m$-invariant and pairwise singular to each other for all $0 \leq j \leq m-1$. By Lemma~\ref{LEM:S^n-ergodic-structure}, there are at most $m$ ergodic measures for $(\mathbb{X},S^m)$, so we obtain that the measures $\mu_j$ are in fact ergodic and that there are precisely $m$ ergodic probability measures on $(\mathbb{X},S^m)$.

Conversely, suppose that there are precisely $m$ ergodic probability measures on $(\mathbb{X},S^m)$. Let $A$ be as in Lemma~\ref{LEM:S^n-ergodic-structure}. The function
\[
f(x) = \sum_{j=0}^{m-1} \me^{2 \pi \mi j/m} \chi^{}_{S^j(A)}(x)
\]
almost surely satisfies $f(Sx) = \me^{2 \pi \mi /m} f(x)$ and is hence a measurable eigenfunction with eigenvalue $\me^{2 \pi \mi/m}$.
\end{proof}

Recall from Definition~\ref{DEF:MSfmdef} that $\MSf(m)$ denotes the number of $S^m$-ergodic Borel probability measures on $\mathbb{X}$.

\begin{lemma}
\label{LEM:eigenvalue-char}
Let $m \in \N$ and let $k$ be the largest divisor of $m$ such that $\me^{2 \pi \mi/k}$ is an eigenvalue. Then, $\MSf(m) =k$.
\end{lemma}

\begin{proof}
If $m_1$ divides $m_2$, then $\MSf(m_2) \geq \MSf(m_1)$. This is because an $S^{m_1}$-invariant measure is also $S^{m_2}$-invariant. Hence, there are at least $\MSf(m_1)$ mutually singular measures that are $S^{m_2}$-invariant.
Given $j,q \in \N$, an $S^{jq}$-ergodic measure $\mu'$ that is $S^q$-invariant is also $S^q$-ergodic. Therefore, if $\MSf(m) = q$, then $\MSf(q) = q$ due to Lemma~\ref{LEM:S^n-ergodic-structure}. By Lemma~\ref{LEM:n-eigenvalues}, $\me^{2 \pi \mi/q}$ is an eigenvalue. It remains to show that $q$ is maximal with this property. Suppose $\ell > q$ is also a divisor of $m$ such that $\me^{2 \pi \mi/\ell}$ is an eigenvalue, implying $\MSf(\ell) = \ell$. Using our first observation in this proof, we obtain $\MSf(m) \geq \MSf(\ell) = \ell > q = \MSf(m)$, a contradiction.
\end{proof}

\begin{lemma}
If $m_1$ and $m_2$ are relatively prime, then $\MSf(m_1m_2) = \MSf(m_1) \MSf(m_2)$. Further, for each prime $p$ there exists a number $\ell_p \in \N_0 \cup \{\infty \}$ such that $\MSf(p^{\ell}) = \min \{ p^{\ell}, p^{\ell_p} \}$ for all $\ell \in \N_0$.
\end{lemma}

\begin{proof}
This follows essentially via the characterization in terms of eigenvalues.
Let $\MSf(m_1m_2) = k$. Since $m_1,m_2$ are coprime, $k$ can be written uniquely as $k = k_1 k_2$ such that $k_1|m_1$ and $k_2|m_2$. Since $\me^{2 \pi \mi/k}$ is an eigenvalue and the eigenvalues build a group, also $\me^{2 \pi \mi/k_1}$ and $\me^{2 \pi \mi/k_2}$ are eigenvalues. If there was an $\ell > k_1$ with $\ell|m_1$ and $\me^{2 \pi \mi /\ell}$ an eigenvalue, then also $\me^{2 \pi \mi/(\ell k_2)}$ would be an eigenvalue. This would imply $\MSf(m_1m_2) \geq \ell k_2 > k$, a contradiction. Hence, $k_1$ is maximal with that property and $\MSf(m_1) = k_1$. Analogously, we find that $\MSf(m_2) = k_2$. This shows the first claim.

Given a prime $p$, let $\ell_p$ be the largest power such that $\me^{2 \pi \mi/p^{\ell_p}}$ is an eigenvalue of $(\mathbb{X}, S, \mu)$. If $\me^{2 \pi \mi / p^{\ell}}$ is an eigenvalue for all $\ell \in \N_0$, set $\ell_p = \infty$. Since the eigenvalues form a group, $\me^{2 \pi \mi / p^{\ell}}$ is also an eigenvalue and hence $\MSf(p^{\ell}) = p^{\ell}$ for all $0 \leq \ell \leq \ell_p$. If $\ell_p\neq \infty$ and $\ell > \ell_p$, the statement $\MSf(p^{\ell}) = p^{\ell_p}$ follows immediately from Lemma~\ref{LEM:eigenvalue-char}.
\end{proof}

\section{Strict Ergodicity of Product Systems} \label{app:StrictErgProd}
\setcounter{subsection}{1}
Here, we pursue the question under which condition the product of two uniquely ergodic/minimal systems is again uniquely ergodic/minimal. 
This is closely related to the joining theory of dynamical systems, pioneered by Furstenberg in \cite{Fu67}. We give an overview of some elementary results for the reader's convenience.

Here, a (topological) \emph{dynamical system} $(X,T)$ consists of a compact metric space $X$ and a homeomorphism $T$ on $X$.

\begin{definition}
A (topological) joining of two dynamical systems $(X_1,T_1)$ and $(X_2,T_2)$ is a non-empty and closed, $T_1 \times T_2$-invariant subset $Z \subseteq X_1 \times X_2$ such that $\pi_1(Z) = X_1$ and $\pi_2(Z) = X_2$, where $\pi_1,\pi_2$ denote the projections to the first and second coordinate, respectively.
\end{definition}

We call $(X_1,T_1)$ and $(X_2,T_2)$ (topologically) \emph{disjoint} if $(X_1 \times X_2, T_1 \times T_2)$ is their only joining.
For the following, compare \cite{Gl03}.

\begin{fact}
Two minimal dynamical systems $(X_1,T_1)$ and $(X_2,T_2)$ are (topologically) disjoint if and only if their product system is minimal.
\end{fact}

There is a natural analogue of this observation in measure-theoretic terms. If $(X,T)$ is a topological dynamical system and $\mu$ a $T$-invariant Borel probability measure on $X$, we call $(X,T,\mu)$ a (measure-preserving) dynamical system.

\begin{definition}
The joining of two measure-preserving dynamical systems $(X_1,T_1,\mu_1)$ and $(X_2,T_2,\mu_2)$ is a $T_1\times T_2$-invariant Borel probability measure $\mu$ on $X_1\times X_2$ such that $\mu_i = \mu \circ \pi_i^{-1}$ for $i \in \{1,2\}$.
\end{definition}

The dynamical systems $(X_1,T_1,\mu_1)$ and $(X_2,T_2,\mu_2)$ are called \emph{disjoint} if $\mu_1 \times \mu_2$ is their only joining.

\begin{lemma}
\label{LEM:disjoint-uniqueness}
Two uniquely ergodic dynamical systems $(X_1,T_1,\mu_1)$ and $(X_2,T_2,\mu_2)$ are disjoint if and only if their product system is uniquely ergodic.
\end{lemma}

\begin{proof}
First, assume that $(X_1,T_1,\mu_1)$ and $(X_2,T_2,\mu_2)$ are not disjoint. Then, there exist at least two different joinings $\mu$ and $\nu$ on $(X_1 \times X_2, T_1 \times T_2)$. By assumption, $\mu$ and $\nu$ are both $T_1 \times T_2$-invariant and hence, the product system is not uniquely ergodic.

Conversely, assume that $(X_1,T_1,\mu_1)$ and $(X_2,T_2,\mu_2)$ are disjoint and let $\mu$ be an arbitrary $T_1 \times T_2$-invariant measure on $X_1 \times X_2$. Since $\mu \circ \pi_i^{-1}$ is $T_i$-invariant for each $i \in \{1,2\}$, the unique ergodicity of $(X_i,T_i, \mu_i)$ implies that $\mu \circ \pi_i^{-1} = \mu_i$. That is, $\mu$ is a joining and as such unique.
\end{proof}

At this point, we have reformulated the original problem in terms of the question under which conditions two minimal/uniquely ergodic dynamical systems are disjoint. 

There is an abundance of useful characterizations and criteria for disjointness; compare for example \cite{Fu67,Gl03,Th95}.
In order to exclude disjointness, it suffices to find a non-trivial common factor of both dynamical systems. More precisely, we call $(Y,S)$ a (topological) factor of $(X,T)$ if there is a continuous surjective map $\pi \colon X \to Y$ such that $\pi \circ T = S \circ \pi$. The factor is called trivial if it coincides with the identity map on a singleton. For the following, see \cite[Prop.~II.2]{Fu67}.

\begin{fact}
\label{FACT:top-factor-not-disjoint}
If $(X_1,T_1)$ and $(X_2,T_2)$ have a non-trivial common topological factor, they are not disjoint.
\end{fact}

There is a natural analogue of this criterion in the measure-theoretic regime. Here, $(Y,S,\nu)$ is called a \emph{factor} of $(X,T,\mu)$ if there is a measurable map $\pi \colon X \to Y$ such that $\pi \circ T = S \circ \pi$ up to null sets, and $\nu = \mu \circ \pi^{-1}$. Such a factor is called trivial it is isomorphic to the identity map on a singleton. The following analogue of Fact~\ref{FACT:top-factor-not-disjoint} can be found in \cite[Prop.~I.2]{Fu67}.

\begin{fact}
\label{FACT:factor-not-disjoint}
If $(X_1,T_1,\mu_1)$ and $(X_2,T_2,\mu_2)$ have a non-trivial common factor, they are not disjoint.
\end{fact}

\begin{remark}
The converse is not true in general \cite{Ru79}. However, the more general statement in \cite[Thm.~8.4]{Gl03} provides a characterization of disjointness in terms of factors and a more general concept, termed \emph{quasifactors}.
\end{remark}

There is a sufficient criterion for disjointness that relies on spectral properties of the dynamical systems \cite[Thm.~6.28]{Gl03}.

\begin{fact}
\label{FACT:spectral-disjointness}
Two dynamical systems $(X_1,T_1,\mu_1)$ and $(X_2,T_2,\mu_2)$ are disjoint if their reduced maximal spectral measures are mutually singular.
\end{fact}

\begin{remark}
Again, the converse of this result is not true. In fact, there are ergodic dynamical systems $(X,T,\mu)$ that are disjoint from their inverse $(X,T^{-1},\mu)$; see for example \cite{BFd,delJ}. On the other hand, the reduced maximal spectral measures of $(X,T,\mu)$ and $(X,T^{-1},\mu)$ are always equivalent; compare \cite{Lem}.
\end{remark}

We are mostly concerned with ergodic dynamical systems. If we further restrict to the class of systems with pure point dynamical spectrum, disjointness has a simple spectral characterization.

\begin{coro}
\label{COR:disjoint-eigenvalues}
Two ergodic systems $(X_1,T_1,\mu_1)$ and $(X_2,T_2,\mu_2)$ with pure point dynamical spectrum are disjoint if and only if they do not have a common eigenvalue except $1$.
\end{coro}

\begin{proof}
First, assume that $1$ is the only shared eigenvalue. Because we assumed that both systems have pure point dynamical spectrum, this implies that the reduced maximal spectral measures are mutually singular, implying disjointness by Fact~\ref{FACT:spectral-disjointness}.

Conversely, assume that both systems share an eigenvalue of the form $\lambda = \me^{2 \pi \im \alpha}$, with $\alpha \in (0,1)$. For a moment, assume that $\alpha$ is irrational. Then, ergodicity implies that the torus translation $R_{\alpha} \colon \T \to \T, x \mapsto x + \alpha$, equipped with the normalized Haar measure is a factor of both systems by standard arguments; compare for example \cite[Lemna~1.6.2]{PF02}. 
If $\alpha \in \Q$, we may assume that $\alpha = 1/r$ for some $r \in \N$ without loss of generality. In this case, the cyclic group $(\Z_r, +1, \nu)$, where $\nu$ is the normalized counting measure is a factor of both systems; see \cite[Lemma~1.6.4]{PF02}. In both cases, the systems cannot by disjoint, due to Fact~\ref{FACT:factor-not-disjoint}.
\end{proof}

Almost-periodic potentials have attracted particular attention in the spectral study of Schr\"odinger operators. Recall that almost-periodic sequences are precisely those that can be obtained from a continuous sampling function $f$ on some minimal group rotation $(\Omega, R)$, where $\Omega$ is a compact metrizable group. 
In fact, the subshift generated by an almost-periodic sequence has itself the structure of such a group rotation.
There is a useful characterization of the strict ergodicity of rotations on a compact metric group \cite[Thm.~1.4.10]{PF02}.
\begin{fact}
\label{FACT:group-rotation-strict-ergodicity}
Let $R \colon \Omega \to \Omega, x \to ax$ be a rotation on a compact metric group $\Omega$. The following are equivalent.
\begin{enumerate}
\item $(\Omega,R)$ is minimal.
\item $(\Omega, R)$ is uniquely ergodic.
\item $\{a^n : n \in \N \}$ is dense in $\Omega$.
\end{enumerate}
In this case, the group $\Omega$ is abelian.
\end{fact}
It is further known that the rotation $R$ on a compact abelian group $\Omega$ has pure point dynamical spectrum \cite[Ch.~1]{PF02}. 

\begin{coro}
\label{COR:strictly-ergodic-product}
Let $(\Omega_1,R_1)$ and $(\Omega,R_2)$ be each a minimal group rotation on a compact metrizable group. Then, the following are equivalent 
\begin{enumerate}
\item $(\Omega_1,R_1)$ and $(\Omega,R_2)$ have no non-trivial shared eigenvalues.
\item $(\Omega_1\times \Omega_2, R_1 \times R_2)$ is uniquely ergodic.
\item $(\Omega_1\times \Omega_2, R_1 \times R_2)$ is minimal.
\end{enumerate}
\end{coro} 

\begin{proof}
The equivalence of (1) and (2) follows by combining Lemma~\ref{LEM:disjoint-uniqueness} with Corollary~\ref{COR:disjoint-eigenvalues}. Since $(\Omega_1\times \Omega_2, R_1 \times R_2)$ is again a group rotation on a compact metrizable group, the equivalence of (2) and (3) follows from Fact~\ref{FACT:group-rotation-strict-ergodicity}.
\end{proof}

\begin{remark}
\label{REM:frequency-module}
Let $R \colon \Omega \to \Omega, x \mapsto ax$ be a minimal group rotation on a compact metrizable group $\Omega$. A \emph{character} on $\Omega$ is a continuous group homomorphism $\chi \colon \Omega \to \mathbb{S}^1 = \{ z \in \C : |z| = 1\}$. We denote by $\widehat{\Omega}$ the dual group of all characters on $\Omega$. In this setup, the characters form a basis of eigenfuctions in $L^2(\Omega,\mu)$, where $\mu$ is the Haar measure on $G$, each with eigenvalue $\chi(a)$. Hence, the group of (topological) eigenvalues of $(\Omega,R)$ is precisely 
\[
G = \{\chi(a) : \chi \in \widehat{\Omega} \}.
\]
The pullback of this group under the projection $\phi \colon \R \to \mathbb{S}^1, t \mapsto \me^{2\pi\im t}$ is often called the \emph{frequency module} of $(\Omega,R)$, denoted by $\mathfrak M$. This object is central in the gap labeling theorem; compare \cite{DFGap, DelSou1983CMP, JohnMos1982CMP}. By construction, the frequency module automatically contains the set $\Z$. We can easily adapt Corollary~\ref{COR:strictly-ergodic-product} to the statement that the product of $(\Omega_1,R_1)$ and $(\Omega_2,R_2)$ is strictly ergodic if and only if their frequency modules have a trivial intersection, given by $\Z$.
\end{remark}

\end{appendix}

\bibliographystyle{abbrvArXiv}

\bibliography{CDFGbib}

\end{document}